\documentclass[12pt]{amsart}
\usepackage{amssymb,amsthm,amsmath, ulem, fullpage}
\usepackage{alltt}

\input xy
\xyoption{all}

\DeclareFontFamily{OMS}{rsfs}{\skewchar\font'60}
\DeclareFontShape{OMS}{rsfs}{m}{n}{<-5>rsfs5 <5-7>rsfs7 <7->rsfs10 }{}
\DeclareSymbolFont{rsfs}{OMS}{rsfs}{m}{n}
\DeclareSymbolFontAlphabet{\scr}{rsfs}

\newtheorem{theorem}{Theorem}[section]
\newtheorem{lemma}[theorem]{Lemma}
\newtheorem{proposition}[theorem]{Proposition}
\newtheorem{corollary}[theorem]{Corollary}

\theoremstyle{definition}
\newtheorem{definition}[theorem]{Definition}
\newtheorem{example}[theorem]{Example}

\theoremstyle{remark}
\newtheorem{remark}[theorem]{Remark}

\newtheorem{question}[theorem]{Question}

\newcommand{\Fr}[2]{F^{#1}_*{#2}}
\newcommand{\Frp}[2]{F^{#1}_*\left({#2}\right)}
\newcommand{\stwo}{$\text{S}_2$}

\renewcommand{\O}{\mbox{$\mathcal{O}$}}
\newcommand{\bm}{\mathfrak{m}}
\newcommand{\ba}{\mathfrak{a}}
\newcommand{\bb}{\mathfrak{b}}
\newcommand{\bQ}{\mathbb{Q}}
\newcommand{\bC}{\mathbb{C}}
\newcommand{\bR}{\mathbb{R}}
\newcommand{\bA}{\mathbb{A}}
\newcommand{\tld}{\widetilde }
\newcommand{\sH}{\scr{H}}
\newcommand{\tensor}{\otimes}
\newcommand{\blank}{\underline{\hskip 10pt}}
\DeclareMathOperator{\Div}{{div}}
\DeclareMathOperator{\Hom}{Hom}
\DeclareMathOperator{\sHom}{{\sH}om}
\DeclareMathOperator{\Ann}{{Ann}}
\DeclareMathOperator{\Spec}{{Spec}}
\DeclareMathOperator{\Supp}{{Supp}}
\DeclareMathOperator{\exc}{{exc}}
\begin{document}

\title{Centers of $F$-purity}
\author{Karl Schwede}

\thanks{The author was partially supported by a National Science Foundation postdoctoral fellowship and by RTG grant number 0502170.}
\address{Department of Mathematics\\ University of Michigan\\ East Hall
530 Church Street \\ Ann Arbor, Michigan, 48109}
\email{kschwede@umich.edu}
\begin{abstract}
In this paper, we study a positive characteristic analogue of the centers of log canonicity of a pair $(R, \Delta)$. We call these analogues \emph{centers of $F$-purity}.  We prove positive characteristic analogues of subadjunction-like results, prove new stronger subadjunction-like results, and in some cases, lift these new results to characteristic zero.  Using a generalization of centers of $F$-purity which we call \emph{uniformly $F$-compatible ideals}, we give a characterization of the test ideal (which unifies several previous characterizations).  Finally, in the case that $\Delta = 0$, we show that uniformly $F$-compatible ideals coincide with the annihilators of the $\mathcal{F}(E_R(k))$-submodules of $E_R(k)$ as defined by Lyubeznik and Smith.
\end{abstract}
\maketitle

\section{Introduction}

In this paper, we study a positive characteristic analogue of \emph{centers of log canonicity}, a notion from algebraic geometry.
Centers of log canonicity are natural objects that appear in the study of the singularities of the minimal model program.  These centers satisfy many (local) geometric properties (see for example \cite{AmbroSeminormalLocus} and \cite{KawamataSubadjunction2}).  We show that the positive characteristic analogue of this notion, which we call \emph{centers of $F$-purity}, satisfies these same geometric properties and, in some cases, vast generalizations of those local properties.  We then use reduction to characteristic $p$ to lift some of these generalizations into characteristic zero.  We also link centers of $F$-purity to a positive characteristic notion that has been studied before, annihilators of $F$-stable submodules of $H^d_{\bm}(R)$ (and their generalizations); also see \cite{LyubeznikSmithCommutationOfTestIdealWithLocalization} and Remark \ref{RemarkLinksWithFESubmodules}.

Consider a pair $(X, \Delta)$ where $X$ is a normal affine scheme and $\Delta$ is an effective $\bQ$-Weil divisor (that is a formal sum of prime Weil divisors with nonnegative rational coefficients) such that $K_X + \Delta$ is $\bQ$-Cartier (that is, some integer multiple is an integral Cartier divisor).  The centers of log canonicity of a pair $(X, \Delta)$ are the (possibly non-closed) points $Q \in X$ where
\begin{itemize}
\item{} for every effective $\bQ$-Cartier divisor $G$ passing through $Q$, the pair $(X, \Delta + G)$ does \emph{NOT} have log canonical singularities at $Q$.
\end{itemize}
In particular, one thinks of the divisor $G$ as having very small (but positive) coefficients.  Roughly speaking, the centers of log canonicity are the points where the singularities of the pair $(X, \Delta)$ are the most severe.  See Section \ref{SubsectionCharacteristicZeroSingularities} for an alternate definition of centers of log canonicity.

In positive characteristic, corresponding to the notion of log canonical pairs, there is a notion of (sharply) $F$-pure pairs $(X, \Delta)$ where $X = \Spec R$, $R$ is an $F$-finite normal domain and $\Delta$ is a $\bQ$-divisor on $X$, \cite{HaraWatanabeFRegFPure}.
Therefore, if one is looking for a positive characteristic analogue of centers of log canonicity for a pair $(X, \Delta)$, it is natural to look for prime ideals $Q \in \Spec R$ such that the pair $(X, \Delta + G)$ is not sharply $F$-pure at $Q$ for any effective $\bQ$-Cartier divisor $G$ that goes through $Q$.  We formulate this condition slightly differently in Section \ref{SectionCentersOfFPurity}.  In fact, we formulate it for triples $(R, \Delta, \ba_{\bullet})$ where $\ba_{\bullet}$ is a graded system of ideals ($\ba_i \cdot \ba_j \subseteq \ba_{i+j}$); see \cite{HaraACharacteristicPAnalogOfMultiplierIdealsAndApplications}.  However, one should note that the results in this paper are interesting even in the case that $\Delta = 0$ and $\ba_{i} = R$ for all $i \geq 0$.

However, once one considers this notion in positive characteristic, it becomes clear that this condition is probably not the right condition to work with (although one can easily show that centers of log canonicity reduce generically, from characteristic zero, to become centers of $F$-purity).  The right condition to consider (formulated in the case that $\Delta = 0$) is ideals $I \subset R$ such that for any map $\phi : R^{1 \over p^e} \rightarrow R$, the submodule $I^{1 \over p^e} \subseteq R^{1 \over p^e}$ is sent into $I$ (that is, $\phi(I^{1 \over p^e}) \subseteq I$), see Section \ref{SectionCOnsistentlyFCompatibleIdeals}.  We call ideals that satisfy this condition \emph{uniformly $F$-compatible}.   The reason for this name is to remind readers of the connection to the notion of compatibly $F$-split ideals of Mehta and Ramanathan, \cite{MehtaRamanathanFrobeniusSplittingAndCohomologyVanishing}.  In particular, for an $F$-pure ring, uniformly $F$-compatible ideals are compatibly $F$-split.  As implied above, the prime uniformly $F$-compatible ideals are precisely the centers of $F$-purity.

There are several ways to characterize uniformly $F$-compatible ideals:
\vskip 6pt
\hskip -12pt
{\bf Lemma \ref{LemmaFSubmoduleLocalCharacterizationOfCenterOfFPurity}, Proposition \ref{PropFedderCriterionForCompatible}. }{ \it Suppose that $(S, \bm)$ is an $F$-finite regular local ring of prime characteristic and that $R = S/I$ is a quotient.  Suppose that $J' \subset S$ is an ideal containing $I$ and that $J = J'/I \subset R$.  Finally let $E_R$ denote the injective hull of the residue field.   Then the following are equivalent:
\begin{itemize}
 \item[(a)]  $J$ is uniformly $F$-compatible.
 \item[(b)]   For every $e > 0$ and every $f \in J$, the composition
\[
\xymatrix@R=3pt{
\Ann_{E_R}(J) = E_{R/J} \ar[r] & E_R \ar[r] & E_R \tensor_R R^{1 \over p^e} \ar[r]^-{\times f^{1 \over p^e}} & E_R \tensor_R R^{1 \over p^e}
}
\]
is zero.
\item[(c)]  For every $e > 0$ we have $(I^{[p^e]} : I) \subseteq (J'^{[p^e]} : J')$.

\end{itemize}}
\vskip 6pt
\hskip -12pt  Characterization (b) generalizes to the contexts of triples $(R, \Delta, \ba_\bullet)$, and characterization (c) generalizes to the context of pairs $(R, \ba_{\bullet})$.

As it turns out, the notion of uniformly $F$-compatible ideals has been studied by Smith and Lyubeznik before, but in a dual context (and not for pairs or triples).  In particular, in the case of a complete local ring $(R, \bm)$ with injective hull of the residue field $E = E(R/\bm)$, uniformly $F$-compatible ideals correspond to the $\mathcal{F}(E)$-submodules of $E$.  For details see \cite[Proposition 5.2]{LyubeznikSmithCommutationOfTestIdealWithLocalization} where they characterize the annihilators of $\mathcal{F}(E)$-submodules of $E$ by condition (c) above.  The notion ``$\mathcal{F}(E)$-submodules of $E$'' is a generalization of Frobenius stable submodules of $H^d_{\bm}(R)$; see \cite{SmithFRatImpliesRat}.

The question of whether there are only finitely many $F$-stable submodules of $H^d_{\bm}(R)$ has been recently studied in \cite{SharpGradedAnnihilatorsOfModulesOverTheFrobeniusSkewPolynomialRing} and also in \cite{EnescuHochsterTheFrobeniusStructureOfLocalCohomology}.  Using their techniques, it immediately follows that there are only finitely many uniformly $F$-compatible ideals associated to a triple $(R, \Delta, \ba_{\bullet})$ if $R$ is a local ring; see Section \ref{SectionRelationsToFStableSubmodulesAndFinitenessOfCenters}.

In characteristic zero, there are a number of theorems related to centers of log canonicity that go under the heading of subadjunction.  In Section \ref{SectionResultsRelatedToFAdjunction} we prove analogues (and in some cases vast generalizations) of these results for centers of $F$-purity (and uniformly $F$-compatible ideals).  In particular, we show that a center of $F$-purity of maximal height (as an ideal) automatically cuts out a strongly $F$-regular scheme; compare with Kawamata's subadjunction theorem, \cite{KawamataSubadjunction2}.  In fact, in a reduced $F$-finite $F$-pure local ring, the center of $F$-purity of maximal height is the splitting prime as defined by Aberbach and Enescu; see \cite{AberbachEnescuStructureOfFPure}.   We also prove the following theorem:
\vskip 6pt
\hskip -12pt
{\bf Theorem \ref{TheoremUnionOfCentersOfFPurityFormAnFPureSubscheme}.  }{\it Suppose that $(R, \Delta, \ba_{\bullet})$ is sharply $F$-pure.  Then any (scheme-theoretic) finite union of centers of sharp $F$-purity for $(R, \Delta, \ba_{\bullet})$ form an $F$-pure subscheme.}
\vskip 6pt
\hskip -12pt
An analogous (but much weaker) result in characteristic zero is that, for a log canonical pair, any (scheme-theoretic) union of centers of log canonicity form a seminormal scheme; see \cite{AmbroSeminormalLocus}.  When we lift Theorem \ref{TheoremUnionOfCentersOfFPurityFormAnFPureSubscheme} to characteristic zero we obtain the following.
\vskip 6pt
\hskip -12pt
{\bf Corollary \ref{CorollaryUnionOfCentersOfLogCanonicityFPureType}. }  {\it Suppose $(X, \Delta)$ is a pair over $\bC$ and $K_X + \Delta$ is $\bQ$-Cartier.  If $(X, \Delta)$ is of dense sharply $F$-pure type, a class of singularities conjecturally equivalent to being (semi-)log canonical, then any (scheme-theoretic) union of centers of log canonicity also has dense $F$-pure type.  In particular, any such union has Du Bois singularities.}
\vskip 6pt
We also show that several common ideals are uniformly $F$-compatible.  In particular, we show that the conductor ideal is uniformly $F$-compatible; see Proposition \ref{PropConductorIsFCompatible}.  A result of Enescu and Hochster implies that annihilators of $F$-stable submodules of $H^d_{\bm}(R)$ are uniformly $F$-compatible; see Remark \ref{RemarkFIdealsAreCompatible} and \cite[Theorem 4.1]{EnescuHochsterTheFrobeniusStructureOfLocalCohomology}.
Furthermore, test ideals (both finitistic and non-finitistic/big; see Section \ref{SectionAppendixOnBigTestIdeals}) are uniformly $F$-compatible.  In fact we can say more:
\vskip 6pt
\hskip -12pt
{\bf Theorem \ref{TheoremTestIdealIsSmallestFCompatible}. }{\it Given a triple $(R, \Delta, \ba_{\bullet})$, the non-finitistic test ideal $\tau_b(\Delta, \ba_{\bullet})$ is the smallest uniformly $(\Delta, \ba_{\bullet}, F)$-compatible ideal whose intersection with $R^{\circ}$ is non-empty.}
\vskip 6pt
\hskip -12pt
In the case where $\Delta = 0$ and $\ba_i = R$ for all $R$, the previous result was proven in \cite{LyubeznikSmithCommutationOfTestIdealWithLocalization}, although it was phrased in the dual language.  On the other hand, in the case that $R$ is regular, $\Delta = 0$ and $\ba_i = \bb^{\lceil t i \rceil}$ (for some ideal $\bb$ and positive real number $t$), the previous result is very closely related to the characterization of the test ideal given in \cite{BlickleMustataSmithDiscretenessAndRationalityOfFThresholds}.
Finally, using the technique of Hara and Watanabe, see \cite{HaraWatanabeFRegFPure}, we also show that multiplier ideals (as well as adjoint ideals and many other related constructions) are uniformly $F$-compatible; see Theorem \ref{TheoremGeneralizationOfCentersOfLogCanonicityAreFCompatible}.

\vskip 12pt \hskip -12pt {\it Acknowledgements: }  I would like to thank Mel Hochster, Mircea Musta{\c{t}}{\u{a}} and Karen Smith for several valuable discussions.  I would also like to thank Shunsuke Takagi for reminding me about the connections with splitting primes and also for explaining how a result of his (which was joint with Nobuo Hara) could be used to relate the non-finitistic test ideal $\tld \tau(\Delta) := \Ann_R 0^{* \Delta}_{E_R}$ with the big test ideal $\tau_b(\Delta)$ of Mel Hochster (see Subsection \ref{SectionAppendixOnBigTestIdeals} for more details).  I would like to thank Professor Jesper Funch Thomsen for pointing out a relevant result in \cite{BrionKumarFrobeniusSplittingMethods}.  Finally, I would like the referee for many useful suggestions and for pointing out many typos.

\section{Definitions and notations}

All rings will be assumed to be noetherian and excellent.  If $R$ is a reduced ring, then $R^{\circ}$ is defined to be the set of elements of $R$ not contained in any minimal primes of $R$.  Unless otherwise specified, rings will be assumed to have characteristic $p$ (that is, they contain a field of characteristic $p > 0$).

If $R$ is a ring of characteristic $p$, we use $F^e : R \rightarrow R$ to denote the $e$-iterated Frobenius (ie. $p^e$th power) map.  We also use ${}^e R$ to denote $R$ viewed as an $R$-module via the $e$-iterated action of Frobenius.  If $R$ is reduced, then ${}^e R$ can also be viewed as $R^{1 \over p^e}$.   If we set $X = \Spec R$ and abuse notation to let $F^e : X \rightarrow X$ to denote the dual map of schemes, then it is harmless to identify ${}^eR$ with (the global sections of) $F^e_* \O_X$. We will usually write $\Fr{e}{R}$ instead of ${}^e R$ (especially when dealing with divisors on $X$).

\begin{definition}
A ring $R$ of characteristic $p > 0$ is called \emph{$F$-finite} if ${}^1 R$ is finite as an $R$-module.
\end{definition}

\begin{remark}
Throughout this paper, \emph{all} rings of positive characteristic will be assumed to be $F$-finite.
\end{remark}

Assume $R$ is a ring of positive characteristic.  If $M$ is a $F^e_* R$-module and if $\ba \subset R$ is an ideal, then we will use $F^e_* \ba M$ to denote the module corresponding to multiplication by the elements of $\ba$, thought of as elements of $F^e_* R$.  Alternately, if we think of  $M$ as an $R^{1 \over p^e}$-module, then $F^e_* \ba M$ corresponds to $\ba^{1 \over p^e} M$.

\subsection{Positive Characteristic Triples}
\label{SubsectionPositiveCharacteristicTriples}
\numberwithin{equation}{subsection}
Suppose $R$ is a ring.  A graded system of ideals $\ba_{\bullet}$ is a collection of ideals $\ba_i$ for integers $i \geq 0$ such that $\ba_i \ba_j \subseteq \ba_{i + j}$.  A common class of graded systems of ideals is $\ba^{\lceil t i\rceil}$ where $t$ is a positive real number.  Notice that
\begin{equation}
\label{EquationRoundingArgument}
\ba^{\lceil t i\rceil} \ba^{\lceil t j \rceil} = \ba^{\lceil t i\rceil + \lceil t j \rceil} \subseteq \ba^{\lceil t(i + j)\rceil}.
\end{equation}
Other important classes are products of the above construction $\Pi_j \ba_j^{\lceil t_j i\rceil}$ and symbolic powers of prime ideals $P$, $\ba_i = P^{(i)}$.

Historically, the second object of a ``pair'' was an effective $\bQ$-Weil divisor $\Delta$ (that is, a formal sum of prime Weil divisors with nonnegative rational coefficients).  Furthermore, it was assumed that $K_X + \Delta$ was $\bQ$-Cartier, a hypothesis that we will not generally need.   For the basics of Weil divisors and reflexive sheaves, please see \cite{HartshorneGeneralizedDivisorsOnGorensteinSchemes}.  If $X = \Spec R$ and $D$ is a divisor on $X$, then we will use $R(D)$ to denote the global sections of $\O_X(D)$.

\begin{definition}
A \emph{triple} $(R, \Delta, \ba_{\bullet})$ is the combined information of a reduced ring $R$, an effective $\bR$-divisor $\Delta$ on $X = \Spec R$ and a graded system of ideals $\ba_i \subseteq R$ satisfying the following conditions:
\begin{itemize}
 \item[(a)] If $\Delta \neq 0$ then $R$ is a normal domain.
 \item[(b)] $\ba_0 = R$.
 \item[(c)] $\ba_i \cap R^{\circ} \neq \emptyset$ for all $i$ (equivalently, $\ba_1 \cap R^{\circ} \neq \emptyset$)
\end{itemize}
If $\Delta = 0$, we will use $(R, \ba_{\bullet})$ to denote $(R, \Delta, \ba_{\bullet})$.  Likewise if $\ba_i = R$ for all $i$, then we will use $(R, \Delta)$ to denote $(R, \Delta, \ba_{\bullet})$.
\end{definition}

\begin{remark}
 Roughly speaking, what condition (a) says is that if we want to work with divisors, then we assume that $R$ is normal.
\end{remark}


\begin{remark}
 While it might seem more complicated to work with triples involving graded systems of ideals $\ba_{\bullet}$, instead of triples $(R,\Delta, \ba^t)$, this greater generality actually simplifies many arguments.  This is because for pairs $(R, \ba^t)$, one repeatedly must make, in various guises, the rounding argument made in equation \ref{EquationRoundingArgument}.
\end{remark}

\begin{definition}
 If $S$ is an $R$-algebra, and $\ba_{\bullet}$ is a graded system of ideals in $R$, then we will use $\ba_{\bullet}S$ to denote the graded system $\{ \ba_i S \}_{i \geq 0}$.
\end{definition}

\begin{definition} \cite{SchwedeSharpTestElements}, \cite[Definition 2.7]{HaraACharacteristicPAnalogOfMultiplierIdealsAndApplications}, \cite{TakagiInversion}, \cite{HaraWatanabeFRegFPure}, \cite{HochsterRobertsFrobeniusLocalCohomology}
\label{DefinitionSharplyFPurePair}
Suppose $R$ is an $F$-finite ring of characteristic $p > 0$.  A triple $(R, \Delta, \ba_{\bullet})$ is called \emph{sharply $F$-pure} if, for a single $e > 0$, there exists an element $d \in \ba_{p^e - 1}$ such that the composition
\[
\xymatrix{
\xi := \Fr{e}{i} \circ \left(\Fr{e}{(\times d)}\right) \circ F^e : R \ar[r] & \Fr{e}{ R} \ar[r]^-{\Fr{e}{(\times d)}} & \Fr{e}{R} \ar[r]& \Fr{e}{R(\lceil (p^e - 1)\Delta \rceil)}
}
\]
splits as a map of $R$-modules.  Here $i$ denotes the inclusion $R \rightarrow R(\lceil (p^{e} - 1)\Delta \rceil)$ and $\Fr{e}{(\times d)}$ denotes the multiplication of $\Fr{e}{R}$-modules by $d$ (ie, $\Fr{e}{(\times d)}$ corresponds to the multiplication by $d^{1 \over p^e}$ if we view $\Fr{e}{R}$ as $R^{1 \over p^e}$).
\end{definition}

\begin{lemma} \cite[Proposition 3.3]{SchwedeSharpTestElements}
\label{LemmaSharplyFPureImpliesInfiniteSplittings}
Suppose that for some $e > 0$ and some $d \in \ba_{p^e-1}$, the map $\xi_e : R \rightarrow \Fr{e}{R(\lceil (p^e - 1)\Delta \rceil)}$, which sends $1$ to $d$, splits.
Then for all positive integers $n$, there exists $d_n \in \ba_{p^{ne} - 1}$ such that the map
\[
R \rightarrow \Fr{ne}{R(\lceil (p^{ne} - 1)\Delta \rceil)}
\]
which sends $1$ to $d_n$, splits.
\end{lemma}
The proof is essentially the same as in \cite{SchwedeSharpTestElements}.
\begin{proof}
Let us use $\psi_e : \Fr{e}{R(\lceil (p^e - 1)\Delta \rceil)} \rightarrow R$ to denote the splitting of $\xi_e$.
Notice that if we apply $\Hom_R(R(-\lceil (p^e - 1)\Delta \rceil), \blank)$ to the map $\psi_e$ we obtain a map
\[
 \xymatrix{
\Fr{e}{R(p^e \lceil (p^e - 1)\Delta \rceil + \lceil (p^e - 1)\Delta \rceil)} \ar[r] & R(\lceil (p^e - 1)\Delta \rceil)
 }
\]
which sends $d$ to $1$ (at least at the level of the total field of fractions), compare with \cite[Proof of Lemma 2.5]{TakagiInterpretationOfMultiplierIdeals}.  Since
\[
p^e \lceil (p^e - 1)\Delta \rceil + \lceil (p^e - 1)\Delta \rceil \geq \lceil (p^{2 e} - 1)\Delta \rceil
\]
we obtain a map
\[
\gamma_e : \Fr{e}{R(\lceil (p^{2e} - 1)\Delta \rceil)} \rightarrow R(\lceil (p^e - 1)\Delta \rceil)
\]
which also sends $d$ to $1$.

We then apply the functor $\Fr{e}{\blank}$ to this map and compose with $\psi_e$.  This gives us a splitting for the map
\[
\xi_{2e} : R \rightarrow \Fr{2e}{R(\lceil (p^{2e} - 1)\Delta \rceil)}
\]
which sends $1$ to $d^{1 + p}$.

By repeating this process
we obtain a splitting for the map $R \rightarrow \Fr{ne}{R(\lceil (p^{ne} - 1)\Delta \rceil)}$ which sends $1$ to $d' = d^{1 + p^e + \dots + p^{(n-1)e}}$.  Therefore, it is sufficient to show that $d' \in \ba_{p^{ne} - 1}$.  However, $d' \in (\ba_{p^e - 1})^{1 + p^e + \dots + p^{(n-1)e}} \subseteq \ba_{(p^e - 1)(1 + p^e + \dots + p^{(n-1)e})} = \ba_{p^{ne} - 1}$.
\end{proof}

\begin{remark}
Given a pair $(R, \Delta)$ as above, one can construct a family of ideals $\ba_{\bullet}$ as follows:  Set $\ba_n = R(-\lceil n\Delta \rceil)$. Note that if $\Delta$ is an integral prime divisor, then $\ba_n$ is simply the $n$th symbolic power of the defining ideal of $\Delta$.  Thus it might be tempting to try to absorb the $\Delta$ into the graded system of ideals $\ba_{\bullet}$; however, I do not know if this is possible.

If the pair $(R, \ba_{\bullet})$, with $\ba_n = R(-\lceil n\Delta \rceil)$, is sharply $F$-pure, then so is the pair $(R, \Delta)$.  In fact, if $(R, \ba_{\bullet})$ is sharply $F$-pure, then for infinitely many $e > 0$, there exists a Cartier divisor $D \geq \lceil (p^e - 1) \Delta \rceil$ such that the map $R \rightarrow \Fr{e}{R(D)}$ splits.  Explicitly, the divisor $D$ is $\Div(d)$ where $d$ is the element $d \in \ba_{p^e - 1}$ along which the map splits.

However, I do not know if the converse holds.  That is, I do not know if $(R, \Delta)$ being sharply $F$-pure implies that the corresponding $(R, \ba_{\bullet})$ is sharply $F$-pure.
\end{remark}

\begin{definition} \cite[Definition 3]{MehtaRamanathanFrobeniusSplittingAndCohomologyVanishing}
\label{DefinitionCompatiblyFSplit}
Suppose that $R$ is a reduced $F$-finite ring and that $I \subseteq R$ is an ideal.  We say that $R$ is \emph{compatibly $F$-split along $I$} if for some $e > 0$ there exists a map $\psi$ such that
\[
\xymatrix{
R \ar[r]^{F^e} & \Fr{e}{R} \ar[r]^-{\psi} & R
}
\]
the composition is the identity and we also have $\psi\left(\Fr{e}{(I)}\right) \subseteq I$.
\end{definition}

\begin{definition} \cite{HaraWatanabeFRegFPure}, \cite{HaraACharacteristicPAnalogOfMultiplierIdealsAndApplications}, \cite{TakagiInversion}
\label{DefinitionStronglyFRegularPair}
Suppose $R$ is an $F$-finite ring of characteristic $p > 0$.  A triple $(R, \Delta, \ba_\bullet)$ is called \emph{strongly $F$-regular} if, for every $c \in R^{\circ}$ there exists $e > 0$ (equivalently infinitely many $e > 0$) and an element $d \in \ba_{p^e - 1}$ such that the map
\[
\xymatrix{
R \ar[r] & \Fr{e}{R(\lceil (p^{e} - 1)\Delta \rceil)}
}
\]
which sends $1$ to $cd$, splits as a map of $R$-modules.
\end{definition}

The following criterion for strong $F$-regularity is also very useful.

\begin{lemma} \cite[Theorem 5.9(a)]{HochsterHunekeFRegularityTestElementsBaseChange}
\label{LemmaSplitAlongAndComplementImpliesFRegular}
Suppose that $(R, \Delta, \ba_{\bullet})$ is a triple. Further suppose there exists an element $c \in R^{\circ}$ such that:
\begin{itemize}
\item[(i)]  There exists an $e > 0$ and some $f \in \ba_{p^e - 1}$ so that the map
\[
\xymatrix{
R \ar[r] & \Fr{e}{R(\lceil (p^{e} - 1)\Delta \rceil)}
}
\]
which sends $1$ to $cf$, splits.
\item[(ii)]  The triple $(R_c, \Delta|_{\Spec R_c}, \ba_{\bullet} R_c)$ is strongly $F$-regular (here $R_c$ is used to denote the localization at $c$).
\end{itemize}
Then $(R, \Delta, \ba_{\bullet})$ is also strongly $F$-regular.
\end{lemma}
\begin{proof}
First note that by hypothesis (i), it is easy to see that the localization map $R \rightarrow R_c = R[c^{-1}]$ is injective.  Now fix $c' \in R^{\circ}$.
By hypothesis (ii), we can find an $e''$ so that there exists $g' \in (\ba_{p^{e''} - 1} R_c)$ giving us a splitting  $\phi : \Fr{e''}{R_c(\lceil (p^{e''} - 1) \Delta|_{\Spec R_c} \rceil )} \rightarrow R_c$ which sends $c' g'$ to $1$.  Notice, that we can write $g'$ as $g' = g/c^{k p^{e''}}$ with $g \in \ba_{p^{e''} - 1}$ so that $\phi$ sends $c' g$ to $c^k$ (all of these elements are viewed in the total field of fractions).  This $\phi$ must be a localization of some $\psi : \Fr{e''}{R(\lceil (p^{e''} - 1) \Delta\rceil)} \rightarrow R$ which sends $c' g$ to $c^m$ for some $m > 0$.  Without loss of generality, since making $m$ larger is harmless, we may assume that $m = p^{(n-1)e} + \dots + p^{e} + 1$ where $e$ is the number guaranteed by hypothesis (i) so that $\psi$ sends $c' g$ to $c^{p^{(n-1)e} + \dots + p^{e} + 1}$.


Note that $h = f^{p^{(n-1)e} + \dots + p^{e} + 1} \in \ba_{p^{ne} - 1}$ as in the proof of Lemma \ref{LemmaSharplyFPureImpliesInfiniteSplittings}.  Since $\psi$ is an $R$-homomorphism, we see that $\psi$ sends $h^{p^{e''}} c' g$ to $h c^{p^{(n-1)e} + \dots + p^{e} + 1} = (fc)^{p^{(n-1)e} + \dots + p^{e} + 1}$.  Now apply the functor $\Hom_R(R(-\lceil (p^{ne} - 1)\Delta \rceil), \blank)$ to $\psi$.  This induces a map
\[
\Psi : \Fr{e''}{R(\lceil (p^{ne + e''} - 1)\Delta \rceil)} \rightarrow \Fr{e''}{R(\lceil (p^{e''} - 1) \Delta\rceil + p^{e''}\lceil (p^{ne} - 1)\Delta \rceil)} \rightarrow R(\lceil (p^{ne} - 1)\Delta \rceil)
\]
which sends $h^{p^{e''}} c' g$ to $h c^{p^{(n-1)e} + \dots + p^{e} + 1}$.

Consider the splitting guaranteed by hypothesis (i).
Via composition, just as in the proof of Lemma \ref{LemmaSharplyFPureImpliesInfiniteSplittings}, we obtain a map $\theta : \Fr{ne}{R(\lceil (p^{ne} - 1)\Delta \rceil)} \rightarrow R$ which sends $(fc)^{p^{(n-1)e} + \dots + p^{e} + 1}$ to $1$.

We then apply $\Fr{ne}{\blank}$ to $\Psi$ and compose with $\theta$.  Therefore, we obtain a map
\[
\xymatrix{
\Fr{e'' + ne}{R(\lceil (p^{ne + e''} - 1)\Delta \rceil)} \ar[r] & \Fr{ne}{R(\lceil (p^{ne} - 1)\Delta \rceil)} \ar[r]^-{\theta} & R
}
\]
which sends $c' g h^{p^{e''}}$ to $1$.  Thus, it is sufficient to show that $g h^{p^{e''}} \in \ba_{p^{ne + e''}-1}$.  But that is easy since
\[
g h^{p^{e''}} \in \ba_{p^{e''}(p^{ne} - 1)} \ba_{p^{e''} - 1} \subseteq \ba_{p^{e''}(p^{ne} - 1) + (p^{e''} - 1)} = \ba_{p^{ne + e''} - 1}.
\]
as desired.
\end{proof}


\begin{example} \cite[Lemma 1.1]{LauritzenRabenThomsenGlobalFRegularityOfSchuertVarieties}, \cite{HaraWatanabeFRegFPure}
\label{ExampleFPurePairsDivisor}
Suppose $(X, \Delta)$ is a pair and that $\Delta$ is a reduced integral divisor.  Further suppose that $(X, \Delta)$ is $F$-pure.  We wish to apply the (natural transformation of) functors
$\sHom_X(\O_X(\Delta), \blank) \rightarrow \sHom_X(\O_X, \blank)$ to a splitting
\[
\xymatrix{
\O_X \ar[r] & \Fr{e}{\O_X} \ar[r] & \Fr{e}{\O_X((p^e - 1)\Delta)} \ar[r] & \O_X.
}
\]
Since $X$ is normal, $\sHom_X(\O_X(\Delta), \O_X) \cong \O_X(-\Delta)$.  Let us now consider the sheaf
\[
\sHom_X(\O_X(\Delta), \Fr{e}{\O_X((p^e - 1)\Delta )}).
\]
Note that
\begin{align*}
\sHom_X(\O_X(\Delta), \Fr{e}{\O_X((p^e - 1)\Delta )}) \cong \Fr{e}{\sHom_X((F^e)^* \O_X(\Delta), \O_X((p^e - 1)\Delta ))} \\
\cong \Fr{e}{\sHom(\O_X(p^e \Delta), \O_X((p^e - 1)\Delta ))} \cong \Fr{e}{\O_X(-\Delta )}
\end{align*}
where the first isomorphism is just adjointness, and the second and third occur because $\O_X((p^e - 1)\Delta )$ is reflexive and all the sheaves agree in codimension 1.  Similarly, we have $\sHom_X(\O_X(\Delta), \Fr{e}{\O_X}) \cong \Fr{e}{\O_X(-p^e \Delta)}$.  These observations, taken together, give us a commutative diagram
\[
\xymatrix{
\O_X(-\Delta) \ar[d] \ar[r] & \Fr{e}{\O_X(-p^e \Delta)} \ar[r] \ar[d] & \Fr{e}{\O_X( -\Delta )} \ar[r] \ar[d] & \O_X(-\Delta) \ar[d] \\
\O_X \ar[r] & \Fr{e}{\O_X} \ar[r] & \Fr{e}{\O_X((p^e - 1)\Delta )} \ar[r] & \O_X
}
\]
where the vertical arrows simply correspond to the natural inclusions of sheaves.  Thus one can conclude that if $(X, \Delta)$ is $F$-pure, then there exists a splitting
\[
\xymatrix{ \O_X(-\Delta) \ar[r] & \Fr{e}{\O_X( -\Delta )} \ar[r]  & \O_X(-\Delta) }
\]
compatible with some splitting of $\O_X$.  Conversely, if there exists such a splitting, then by applying $\sHom_X(\O_X(-\Delta), \blank)$ we can conclude that $(X, \Delta)$ is $F$-pure.
\end{example}

\subsection{``Big'' test ideals}
\label{SectionAppendixOnBigTestIdeals}

In a course taught at the University of Michigan in Fall 2007, Mel Hochster worked out the theory of ``big test elements'' and ``big test ideals''.  Roughly speaking, the big test elements are those elements of $R^{\circ}$ that multiply the (non-finitistic) tight closure of any module into itself.  From one point of view, this theory has been explored by Smith and Lyubeznik \cite{LyubeznikSmithCommutationOfTestIdealWithLocalization} and also by Takagi and Hara \cite{HaraTakagiOnAGeneralizationOfTestIdeals} when they were studying the object $\tld \tau(R) := \Ann_R(0^{*}_{E_R})$ but I do not know of a reference where Hochster's point of view is worked out.

\begin{definition}  \cite{TakagiInterpretationOfMultiplierIdeals}, \cite{HaraYoshidaGeneralizationOfTightClosure}, \cite{HaraACharacteristicPAnalogOfMultiplierIdealsAndApplications}, \cite{HochsterHunekeTC1}
Suppose $R$ is an $F$-finite reduced ring, $X = \Spec R$ and $(X, \Delta, \ba_{\bullet})$ is a triple.  Further suppose that $M$ is a (possibly non-finitely generated) $R$-module and that $N$ is a submodule of $M$.  We say that an element $z \in M$ is in the \emph{$(\Delta, \ba_{\bullet})$-tight closure of $N$ in $M$}, denoted $N^{*\Delta, \ba_{\bullet} }_M$, if there exists an element $c \in R^{\circ}$ such that, for all $e \gg 0$ and all $a \in \ba_{p^e - 1}$, the image of $z$ via the map
\[
\xymatrix{
\Fr{e}{i} \circ \Fr{e}{(\times c a)} \circ \ F^e : M \ar[r] & M \tensor_R \Fr{e}{R} \ar[r]^-{\Fr{e}{(\times c a)}} & M \tensor_R \Fr{e}{R} \ar[r] & M \tensor_R \Fr{e}{R(\lceil (p^e - 1)\Delta \rceil)} \\
}
\]
is contained in $N^{[q]\Delta}_M$, where we define $N^{[q] \Delta}_M$ to be the image of $N \tensor_R \Fr{e}{R(\lceil (p^e - 1)\Delta \rceil)}$ inside $M \tensor_R \Fr{e}{R(\lceil (p^e - 1)\Delta \rceil)}$.  For convenience, we use $c a z^{p^e}$ to denote $\Fr{e}{i} \circ \Fr{e}{(\times ca)} \circ  F^e(z)$.
\end{definition}

\begin{remark}
 Notice that the composition in the previous definition is obtained by simply taking the map $\xi : R \rightarrow \Fr{e}{R(\lceil (p^e - 1)\Delta \rceil)}$ which sends $1$ to $ca$, and then tensoring with $M$.
\end{remark}

\begin{definition}\cite[Lecture of September 17th]{HochsterFoundations}
An element $c \in R$ is called a \emph{big sharp test multiplier} for the triple $(R, \Delta, \ba_{\bullet})$ if for all $u \in N^{*\Delta, \ba_{\bullet}}_M$, all $e \geq 0$ and all $a \in \ba_{p^e - 1}$ one has $c a z^{p^e} = (\Fr{e}{i} \circ \Fr{e}{(\times ca )} \circ  F^e)(u) \in N^{[p^e] \Delta}_M$.

We say that such a $c$ is a \emph{big sharp test element} if, in addition, $c \in R^{\circ}$.
\end{definition}

We now show that big sharp test elements in this generality always exist, see \cite[Theorem 5.9]{HochsterHunekeFRegularityTestElementsBaseChange}, \cite{HaraYoshidaGeneralizationOfTightClosure} and \cite{TakagiInterpretationOfMultiplierIdeals}.

\begin{lemma}
\label{LemmaBigSharpTestElementsExist}
If $(R, \Delta, \ba_{\bullet})$ is a triple, then there exists a big sharp test element for the triple.
\end{lemma}
\begin{proof}
Because $R$ is reduced, we can choose an element $c \in R^{\circ}$ such that the localized triple $(R_c, \Delta|_{\Spec R_c}, (\ba R_c)_{\bullet})$ is regular (in particular $(\ba R_c)_{i} = R_c$ for all $i$).  Then, using the standard arguments on the existence of test elements, for all $d \in R^{\circ}$ and all sufficiently large and divisible $k$, there exists $a_{k} \in \ba_{p^{k} - 1}$ and a map $\phi : F^{k}_* R(\lceil (p^{k} - 1) \Delta \rceil) \rightarrow R$ such that $\phi(da_{k}) = c^m$ for some fixed $m$ (note that $m$ does not depend on $d$).

Now suppose that $z \in N^{* \Delta, \ba_{\bullet}}_M$.  Thus we know that there exists $d \in R^{\circ}$ and $k > 0$ such that $d \ba_{p^{e+k} - 1}z^{p^{e+k}} \subset N^{[p^{e+k}] \Delta}_M$ for all $e \geq 0$.  Without loss of generality, we may assume the $k$ gives a map $\phi$ (such that $\phi(da_{k}) = c^m$ for some $a_{k} \in \ba_{p^{k} - 1}$) for this fixed $d$ as in the first paragraph of the proof.

Now, for each $e > 0$, $\phi$ induces a map
\[
\phi_{e} : \Fr{e+k}{R(\lceil (p^{e+k} - 1) \Delta \rceil )} \rightarrow \Fr{e}{R(\lceil (p^e - 1) \Delta \rceil)}
\]
such that $\phi(da_{k}) = c^m$ at the level of the total field of fractions.  We can tensor these maps with both $M$ and $N$ to obtain the following $\Fr{e}{R}$-linear diagram of maps
\[
\xymatrix{
 N^{[p^{e+k}] \Delta}_M \ar@{^{(}->}[d] \ar[r] & N^{[p^{e}] \Delta}_M \ar@{^{(}->}[d]\\
 M \tensor_R \Fr{e+k}{R(\lceil (p^{e+k} - 1) \Delta \rceil )} \ar[r]^-{\psi_e} & M \tensor_R \Fr{e}{R(\lceil (p^e - 1) \Delta \rceil)} \\
}
\]
Notice that we labeled the bottom horizontal arrow $\psi_e$.

Now, notice that
\[
d \ba_{p^e -1}^{[p^{k}]} a_{k} z^{p^{e+k}} \subseteq d \ba_{p^e -1}^{p^{k}} \ba_{p^{k} - 1} z^{p^{e+k}} \subseteq d \ba_{p^{e+k} - 1}z^{p^{e+k}} \subset N^{[p^{e+k}] \Delta}_M
\]
Therefore, when we apply $\psi_e$, we see that $c^m \ba_{p^e - 1} z^{p^e} = \psi_e(d \ba_{p^e -1}^{[p^{k}]} a_{k} z^{p^{e+k}}) \subseteq N^{[p^{e}] \Delta}_M$ for all $e \geq 0$ by the commutativity of the above diagram.  This implies that $c^m$ is a big sharp test element, as desired.
\end{proof}

\begin{remark}
\label{RemarkLocalizingTestElementsExist}
In the above proof, we can certainly choose $c$ such that $c \in \ba_1 \cap R^{\circ}$ and such that the support of $\Div(c)$ contains the support of $\Delta$.  In that case, $R_c$ is regular (and the terms $\ba_{\bullet}$ and $\Delta$ essentially go away).  I then claim that $c^m/1$ is still a big sharp test element after localization at any multiplicative subset $W$.

Now, for any $d \in (W^{-1} R)^{\circ}$, it is well known to experts that there is a $d' \in R^{\circ}$ such that $d'/s = d$ for some $s \in W$ (see for example \cite[Page 57]{HochsterFoundations}).  Alternately, if $R$ is a domain, then the statement is obvious.  Therefore, for each $d \in (W^{-1} R)^{\circ}$, at least in terms of testing tight closure, we may assume that $d = d'/1$ for $d' \in R^{\circ}$.  Thus, we may take the map $\phi : F^{k}_* R(\lceil (p^{k} - 1) \Delta \rceil) \rightarrow R$ such that $\phi(da_{k}) = c^m$ for some $a_{k} \in \ba_{p^{k} - 1}$ and then localize at the multiplicative subset $W$.  The proof that $c^m/1$ is a big sharp test element for $(W^{-1}R, \Delta|_{\Spec W^{-1} R}, (\ba W^{-1} R)_{\bullet})$ then follows just as above.
\end{remark}

\begin{definition}
The set of sharp test multipliers for the triple $(R, \Delta, \ba_{\bullet})$ form an ideal, which we call the \emph{big test ideal}, and denote by $\tau_b(\Delta, \ba_{\bullet})$.
\end{definition}

As mentioned before, there is another non-finitistic ``test-ideal'' which has been heavily studied in the last decade; see for example \cite{LyubeznikSmithCommutationOfTestIdealWithLocalization} or \cite{HaraTakagiOnAGeneralizationOfTestIdeals}.  This test ideal, usually denoted by $\tld \tau(\Delta, \ba_{\bullet})$, is obtained by taking the annihilator of the $(\Delta, \ba_{\bullet})$-tight closure of zero in the injective hull of the residue field (assuming $R$ is a local ring).  In \cite[Lecture of November 30th]{HochsterFoundations}, Hochster shows that $\tld \tau(R) \cong \tau_b(R)$ (also compare with \cite[Proposition 8.23(c)]{HochsterHunekeTC1}).  However, Hochster's approach doesn't work in the case of triples $(R, \Delta, \ba_{\bullet})$, so techniques of Takagi will be employed instead.  One should note that Hochster's argument does go through essentially unchanged for pairs $(R, \ba_{\bullet})$

We would like to prove that $\tau_b(\Delta, \ba_{\bullet}) = \tld \tau(\Delta, \ba_{\bullet}) := \Ann_R(0^{*\Delta, \ba_{\bullet}}_{E_R}$).  The difficulty is that when working with tight closure of pairs $(X, \Delta)$, one does not have an analogue of the easy fact $(N^{[p^e]}_M)^{[p^{e'}]}_{M \tensor_R \Fr{e}{R}} \cong N^{[p^{e+e'}]}_M$ (there is a map, but it goes the ``wrong way'' for our purposes).  Shunsuke Takagi has explained to me an alternate way around this problem, and what is recorded for the rest of the subsection is his argument.

First we recall the following result of Hara and Takagi.
\numberwithin{equation}{theorem}
\begin{lemma}\cite[Lemma 2.1]{HaraTakagiOnAGeneralizationOfTestIdeals}
\label{LemmaHaraTakagiCharacterizationOfTestIdeal}
 Suppose that $(R, \bm)$ is an $F$-finite local ring of characteristic $p > 0$ and that $X = \Spec R$.  Further suppose that $(X, \Delta, \ba_{\bullet})$ is a triple.  Fix a system of generators $x_1^{(e)}, \ldots, x_{r_e}^{(e)}$ of $\ba_{p^e - 1}$ for each $ e \geq 0$.  Then an element $c \in R$ is contained in $\tld \tau(\Delta, \ba_{\bullet}) := \Ann_R(0^{*\Delta, \ba_{\bullet}}_{E_R})$ if and only if, for some (equivalently any) big sharp test element $d \in R^{\circ}$ there exists an integer $e_1 > 0$ and there exist $R$-homomorphisms $\phi_i^{(e)} \in \Hom_R(\Fr{e}{R\lceil(p^e - 1) \Delta\rceil}, R)$ for $0 \leq e \leq e_1$ and $1 \leq i \leq r_e$ such that
\begin{equation}
\label{EquationHaraTakagi}
 c = \sum_{e = 0}^{e_1} \sum_{i = 1}^{r_e} \phi_i^{(e)}(d x_i^{(e)}).
\end{equation}
\end{lemma}
\begin{proof}
The proof is the same as \cite[Lemma 2.1]{HaraTakagiOnAGeneralizationOfTestIdeals}, also compare with \cite{TakagiPLTAdjoint}.
\end{proof}

\begin{remark}
\label{RemarkDefinitionInNonlocal}
Fix a multiplicative system $W$.  By Remark \ref{RemarkLocalizingTestElementsExist}, we can construct a big sharp test element $d$ for $(R, \Delta, \ba_{\bullet})$ such that $d/1$ is simultaneously a big sharp test element for $(W^{-1} R, \Delta|_{\Spec W^{-1} R}, (\ba (W^{-1} R))_{\bullet})$.  Thus, working with that big sharp test element and the above lemma, it is easy to see that the construction of the ideal $\tld \tau(\Delta, \ba_{\bullet})$ commutes with localization since there are only finitely many conditions to check and since $\Hom$ commutes with localization.
This is what is done in \cite[Proposition 3.1]{HaraTakagiOnAGeneralizationOfTestIdeals} (note one does need the existence of some form of test elements that localize).

Therefore, it is natural to use the characterization of the previous lemma to define $\tld \tau(\Delta, \ba_{\bullet})$ in the case that $R$ is not a local ring (at least assuming you fix $d$ to be a big sharp test element as constructed in Remark \ref{RemarkLocalizingTestElementsExist} and Lemma \ref{LemmaBigSharpTestElementsExist}).
\end{remark}

We are now in a position to explain Takagi's argument relating $\tau_b(\Delta)$ with $\tld \tau(\Delta)$.  Because I do not know of a reference where this is published, I include the argument below.  Also compare with \cite[Lecture of November 30th]{HochsterFoundations} and \cite[Proposition 8.23(c)]{HochsterHunekeTC1}.

\begin{theorem}
\label{TheoremBigVsNonFinitisticTestIdeals}
 Suppose that $R$ is a ring of positive characteristic, and that $(R, \Delta, \ba_{\bullet})$ is a triple, in particular $R$ is $F$-finite.  Then $\tau_b(\Delta, \ba_\bullet) = \tld \tau(\Delta, \ba_\bullet)$.
\end{theorem}
The following argument is due to Shunsuke Takagi.
\begin{proof}
Since the containment $\subseteq$ is obvious, let us choose $c \in \tld \tau(\Delta, \ba_{\bullet})$ and also fix $e \geq 0$.  Fix $N \subseteq M$ to be a containment of $R$-modules and suppose that $z \in N^{* \Delta, \ba_{\bullet}}_M$.  Thus for any big sharp test element $d \in R^{\circ}$, we have that $d a z^{p^{e'}} \in N^{[p^{e'}] \Delta}_M$ for all $e' \geq 0$ and all $a \in \ba_{p^{e'} - 1}$.  Fix a system of generators $x_1^{(k)}, \ldots, x_{r_k}^{(k)}$ of $\ba_{p^k - 1}$ for each $k \geq 0$.  By Lemma \ref{LemmaHaraTakagiCharacterizationOfTestIdeal}, there exists $e_1 > 0$ and maps $\phi_i^{(k)} \in \Hom_R(\Fr{k}{R (\lceil (p^k - 1) \Delta \rceil)}, R)$ for $0 \leq k \leq e_1$ and $1 \leq i \leq r_k$ such that
\[
c = \sum_{k = 0}^{e_1} \sum_{i = 1}^{r_k} \phi_i^{(k)}(d x_i^{(k)}).
\]
Each $\phi_i^{(k)}$ induces a map
\[
\Fr{e+k}{R(\lceil (p^{e+k} - 1) \Delta \rceil )} \rightarrow \Fr{e}{R(\lceil (p^e - 1) \Delta \rceil)},
\]
compare with \cite[Proof of Lemma 2.5]{TakagiInterpretationOfMultiplierIdeals}.  We can tensor this map with both $M$ and $N$ to obtain the following $\Fr{e}{R}$-linear diagram of maps
\[
\xymatrix{
 N^{[p^{e+k}] \Delta}_M \ar@{^{(}->}[d] \ar[r] & N^{[p^{e}] \Delta}_M \ar@{^{(}->}[d]\\
 M \tensor_R \Fr{e+k}{R(\lceil (p^{e+k} - 1) \Delta \rceil )} \ar[r]^-{\psi_i^{(k,e)}} & M \tensor_R \Fr{e}{R(\lceil (p^e - 1) \Delta \rceil)} \\
}
\]
Notice that we labeled the bottom horizontal arrow $\psi_i^{(k,e)}$.

Now fix an arbitrary $a \in \ba_{p^e - 1}$ and note that $\psi_i^{(k,e)}$ sends $d x_i^{(k)} a^{p^k} z^{p^{e+k}}$ to $\phi_i^{(k)}(d x_i^{(k)}) a z^{p^e}$.  But also observe that
\[
x_i^{(k)} a^{p^k} \in \ba_{p^k - 1} \ba_{p^e - 1}^{p^k} =  \ba_{p^k - 1} \ba_{p^{e + k} - p^k} \subseteq \ba_{p^{e + k} - 1}.
\]
Thus, $d x_i^{(k)} a^{p^k} z^{p^e+k} \in N^{[p^{e+k}] \Delta}_M$ since $e + k \geq 0$.  By the commutativity of the above diagram, this implies that
\[
\phi_i^{(k)}(d x_i^{(k)}) a z^{p^e} \in N^{[p^{e}] \Delta}_M.
\]
Therefore, we obtain that
\[
 c a z^{p^e} = \sum_{k={0}}^{e_1} \sum_{i = 1}^{r_k} \phi_i^{(k)}(d x_i^{(k)}) a z^{p^e} \in N^{[p^e] \Delta}_M \text{ for all $e \geq 0$ }
\]
as desired.
\end{proof}

We also obtain the following corollary (which is straightforward to prove directly in the case of a pair $(R, \ba_{\bullet})$ or in the classical (non-pair) setting).

\begin{corollary}
 Suppose that $R$ is a ring, $X = \Spec R$ and that $(X, \Delta, \ba_{\bullet})$ is a pair.  Then $\tau_b(\Delta, \ba_{\bullet}) = \bigcap_{N \subseteq M} (N : N^{*\Delta, \ba_{\bullet}}_M)$ where the intersection runs over all $R$-modules $N \subseteq M$.
\end{corollary}
\begin{proof}
It is clear that $\tau_b(\Delta, \ba_{\bullet}) \subseteq \bigcap_{N \subseteq M} (N : N^{*\Delta, \ba_{\bullet}}_M)$.  On the other hand, $\bigcap_{N \subseteq M} (N : N^{*\Delta, \ba_{\bullet}}_M) \subseteq \tld \tau(\Delta, \ba_{\bullet})$.
\end{proof}

\begin{remark}
Using the same arguments as in \cite[Proposition 3.1, Proposition 3.2]{HaraTakagiOnAGeneralizationOfTestIdeals} (at least once one has a big sharp test element that is still a big sharp test element after localization and completion), one sees that the formation of $\tau_b(\Delta, \ba_{\bullet})$ whose formation commutes with localization and completion.
\end{remark}

\subsection{Characteristic zero singularities}
\label{SubsectionCharacteristicZeroSingularities}

In this subsection, we give a brief description of log canonical singularities and centers of log canonicity.  For a more complete introduction, please see \cite{KollarSingularitiesOfPairs} or \cite{KollarMori}.  We will work with pairs $(X = \Spec R, \Delta)$ in the characteristic zero setting.

Suppose that $(X, \Delta)$ is a pair such that $X$ is a normal affine scheme of finite type over a field of characteristic zero and $K_X + \Delta$ is $\bQ$-Cartier.  By a \emph{log resolution of $(X, \Delta)$}, we mean a proper birational morphism $\pi : \tld X \rightarrow X$ that satisfies the following conditions.
\begin{itemize}
\item[(i)]  $\pi$ is proper and birational,
\item[(ii)]  $\tld X$ is smooth,
\item[(iii)]  The support of the divisor $\pi^{-1}_* \Delta$ is a smooth subscheme of $\tld X$ (see \cite{KollarMori} for more discussion of $\pi^{-1}_* \Delta$, the \emph{strict transform of $\Delta$}),
\item[(iv)]  The divisor $\Supp(\pi^* \Delta) \cup \exc(\pi)$ has simple normal crossings.
\end{itemize}
Log resolutions always exist if $X$ is of finite type over a field of characteristic zero by \cite{HironakaResolution}.
After picking such a log resolution $\pi : \tld X \rightarrow X$, we can fix a canonical divisor $K_{\tld X}$ on $\tld X$ and we may assume (by changing $K_X$ within its linear equivalence class if needed) that $\pi_* K_{\tld X} = K_X$.  Since $K_X + \Delta$ is $\bQ$-Cartier, there exists an integer $m > 0$ such that $m(K_X + \Delta)$ is Cartier, and so we use the notation $\pi^*(K_X + \Delta)$ to denote the $\bQ$-divisor ${1 \over m} \pi^*(m(K_X + \Delta))$.  Thus we can write
\[
K_{\tld X} = \pi^*(K_X + \Delta) + \sum_i^n a_i E_i
\]
where the $E_i$ are exceptional divisors or components of the strict transform $\pi^{-1}_* \Delta$, and the $a_i$ are rational numbers.  The $\bQ$-divisor $\sum_i^n a_i E_i$ is independent of the choice of representative for $K_{\tld X}$.

\begin{definition}
We say that a pair $(X, \Delta)$ (as above, with $K_X + \Delta$ $\bQ$-Cartier) has \emph{log canonical singularities} if, for every (equivalently for a single) log resolution, the numbers $a_i$, which are called \emph{discrepancies}, are all bigger than or equal to $-1$.  If all the discrepancies are strictly bigger than $-1$ (which implies that $\lfloor \Delta \rfloor = 0$), then we say that $(X, \Delta)$ is \emph{Kawamata log terminal}.
\end{definition}

\begin{definition}
\label{DefinitionCenterOfLogCanonicity}
Given a pair $(X, \Delta)$ as above, we say that a (possibly non-closed) point $P \in X$ is a \emph{center of log canonicity for $(X, \Delta)$} (or equivalently, a \emph{log canonical center for $(X, \Delta)$}) if there exists a log resolution $\pi$ and an divisor $E_i$ on $X$ dominating $P$ and such that the coefficient $a_i$ of $E_i$ satisfies $a_i \leq -1$.
\end{definition}

\begin{remark}
As stated in the introduction, it is not difficult to see that a (possibly non-closed) point $Q \in X = \Spec R$ is center of log canonicity for $(X, \Delta)$ if and only if for every $f \in R^{\circ}$ and every rational $\epsilon > 0$, one has that the localized pair  $(X_Q, \Delta_Q + \epsilon \Div(f)_Q)$, is not log canonical.  It is this characterization that we will use to inspire our characteristic $p > 0$ analogue of centers of log canonicity.
\end{remark}



\begin{remark}
 Notice that $(X, \Delta)$ is Kawamata log terminal if and only if it contains no centers of log canonicity.
\end{remark}

\begin{remark}
In many cases, centers of log canonicity are only defined for log canonical pairs, in which case the requirement of Definition \ref{DefinitionCenterOfLogCanonicity} reduces to $a_i = -1$ (instead of $a_i \leq -1$).
\end{remark}

\begin{remark}
A large amount of the interest in these singularities has been under the hypothesis that $X$ is projective (and not affine).  In this paper however, we restrict ourselves to the affine case.
\end{remark}

There is a further generalization of log canonical singularities (without the normality hypothesis) that we will need.

\begin{definition} \cite{KollarShepherdBarron}, \cite{KollarFlipsAndAbundance}, \cite{KovacsSchwedeSmithLCImpliesDuBois}
Let $X$ be an reduced affine seminormal \stwo{} scheme and set $D$ to be the effective divisor of the conductor ideal on $X^N$ (the normalization of $X$).  Note that $D$ is a reduced divisor by \cite[Lemma 7.4]{GrecoTraversoSeminormal} and \cite[Lemma
  1.3]{TraversoPicardGroup}.  We say that $X$ has \emph{weakly semi-log canonical} singularities if the pair $(X^N, D)$ is log canonical.
\end{definition}

Finally, we define a class of singularities in characteristic zero by reduction to characteristic $p > 0$.  See \cite{HochsterRobertsFrobeniusLocalCohomology}, \cite{HochsterHunekeTightClosureInEqualCharactersticZero} and \cite{HaraWatanabeFRegFPure} for discussions of the process of reduction to characteristic $p > 0$.

\begin{definition}
We say that $(X, \Delta)$ has \emph{dense sharply $F$-pure type} if, after reduction to a family of characteristic $p > 0$ models $(X_p, \Delta_p)$ of $(X, \Delta)$, a Zariski-dense set of those models are sharply $F$-pure.
\end{definition}




\section{Uniformly $F$-compatible ideals}
\label{SectionCOnsistentlyFCompatibleIdeals}

In this section, we consider ideals of $R$ that are sent back into themselves via every map (or, in the case of pairs, a certain collection of maps) $\Fr{e}{R} \rightarrow R$.  We call these ideals \emph{uniformly $F$-compatible} in order to remind readers of the connection with compatibly $F$-split subschemes (ideals) as defined by Mehta and Ramanathan; see \cite{MehtaRamanathanFrobeniusSplittingAndCohomologyVanishing} and also Definition \ref{DefinitionCompatiblyFSplit}.  We will also see that these objects are essentially dual to the $\mathcal{F}(E_R(k))$-submodules of $E_R(k)$ as defined in \cite{LyubeznikSmithCommutationOfTestIdealWithLocalization}.  We will not rely on Lyubeznik and Smith's of view for the remainder of the paper, but we will point out relations as they appear.  In particular, some statements in the following section will partially overlap with results of \cite{LyubeznikSmithCommutationOfTestIdealWithLocalization}.

\begin{definition}
\label{DefinitionUniformlyFCompatible}
Suppose $R$ is $F$-finite and that, $(R, \Delta, \ba_{\bullet})$ is a triple.   We say that $I$ is \emph{uniformly $(\Delta, \ba_{\bullet}, F)$-compatible} if, for every $e > 0$, every $a \in \ba_{p^e-1}$ and every $R$-linear map
\[
\phi : \Fr{e}{R( \lceil (p^e - 1) \Delta \rceil) } \rightarrow R,
\]
we have $\phi(\Fr{e}{\left( a I \right)} ) \subseteq I$ (equivalently, we can require that $\phi(F^e_* \ba_{p^e - 1} I) \subseteq I$).  Notice that $\Fr{e}{\left( a I \right)} \subseteq \Fr{e}{R} \subseteq \Fr{e}{R( \lceil (p^e - 1) \Delta \rceil) }$.  If $\ba_i = R$ for every $i \geq 0$ and $\Delta = 0$, then we say simply that $I$ is \emph{uniformly $F$-compatible}.
\end{definition}


\begin{remark}
In the case that $\Delta = 0$ and $\ba_i = R$ for all $i \geq 0$, uniformly $F$-compatible ideals can be thought of as the ideals $I \subset R$ such that every map $\Fr{e}{R} \rightarrow R$ induces a map $\Fr{e}{R/I} \rightarrow R/I$.
\end{remark}

\begin{corollary}
\label{CorUniformlyFCompatibleDefinesFPureSubscheme}
If $(R, \Delta, \ba_{\bullet})$ is sharply $F$-pure and $I$ is uniformly $(\Delta, \ba_{\bullet},F)$-compatible, then $I$ defines an $F$-pure subscheme (and in particular $I = \sqrt{I}$).
\end{corollary}
\begin{proof}
Fix some $e > 0$, an element $a \in \ba_{p^e-1}$ and fix a map $\phi : \Fr{e}{R(\lceil (p^e - 1) \Delta \rceil)} \rightarrow R$ such that $\phi(a) = 1$. We have the following diagram with split rows:
\[
\xymatrix{
I \ar[r] \ar[d] & \Fr{e}{I} \ar[rrr] \ar[d] & & & I \ar[d] \\
R \ar[r] \ar[d] & \Fr{e}{R }\ar[r]^-{\Fr{e}{(\times a)}} \ar[d] &  \Fr{e}{R} \ar[r] & \ar[r]^-{\phi} \Fr{e}{R(\lceil (p^e - 1) \Delta \rceil)}  & R \ar[d] \\
R/I \ar[r] & \Fr{e}{R/I} \ar[rrr] & & & R/I
}
\]
This implies that $R/I$ is $F$-pure, and in particular that $I$ is a radical ideal.
\end{proof}

\begin{remark}
It might be tempting to try to conclude that $(R/I, (\ba_{\bullet} R/I))$ is a sharply $F$-pure pair, but that does not follow since the map $\phi$ in the proof above does not necessarily send $\Fr{e}{I}$ into $I$.  However, if $I$ is also uniformly $F$-compatible simply for the ring $R$ (without any pair or triple), and $(R, \ba_{\bullet})$ is sharply $F$-pure, then it follows that $(R/I, (\ba_{\bullet} R/I))$ is sharply $F$-pure.  Similar statements can be made for triples $(R, \Delta, \ba_{\bullet})$.
\end{remark}

We also have the following lemmas, which will be useful later.

\begin{lemma}
\label{EasyPropertiesOfUniformlyFCompatibleIdeals}
Suppose $R$ is $F$-finite, $(R,  \Delta,\ba_{\bullet})$ is a triple and that $\{I_i\}$ is a collection of uniformly $(\Delta, \ba_{\bullet},F)$-compatible ideals.  Then $\bigcap_i I_i$ and $\Sigma_i I_i$ are uniformly $(\Delta, \ba_{\bullet}, F)$-compatible.
\end{lemma}
\begin{proof}
Fix a map
\[
\phi : \Fr{e}{R(\lceil (p^e - 1) \Delta \rceil)} \rightarrow R.
\]
First choose $x \in \bigcap_i I_i$.  Then by assumption, for every $a \in \ba_{p^e-1}$, we have $\phi(a x) \in I_i$.  This implies that for an arbitrary (but fixed) $a \in \ba_{p^e - 1}$, we have $\phi(ax) \in \bigcap_i I_i$.  Likewise, consider $x_{i_1} + \ldots + x_{i_n} \in \Sigma_i I_i$ (where $x_{i_j} \in I_{i_j}$).  Then $\phi(a(x_{i_1} + \ldots + x_{i_n})) = \phi(ax_{i_1}) + \ldots + \phi(ay_{i_n}) \in \Sigma_i I_i$ as desired.
\end{proof}

It is also not difficult to see that uniformly $F$-compatible ideals localize well; compare with \cite[Proposition 5.3]{LyubeznikSmithCommutationOfTestIdealWithLocalization}.

\begin{lemma}
\label{UniformlyFCompatibleIdealsLocalize}
Suppose that $R$ is $F$-finite, $I \subseteq R$ is an ideal and $(R, \Delta, \ba_{\bullet})$ is a triple.  Further suppose that $T \subset R$ is a multiplicative system. If $I$ is uniformly $(\Delta, \ba_{\bullet}, F)$-compatible, then $T^{-1} I \subseteq T^{-1}R$ is uniformly $(\Delta|_{\Spec T^{-1}R}, (T^{-1} \ba)_{\bullet}, F)$-compatible.
\end{lemma}
\begin{proof}
Set $S = T^{-1} R$ and set $\Delta' = \Delta|_{\Spec S}$.
Consider a map $\psi :  \Fr{e}{S(\lceil (p^e - 1)\Delta' \rceil)} \rightarrow S$.  Since $R$ is $F$-finite, we know that there exists some map $\phi : \Fr{e}{R(\lceil (p^e - 1)\Delta \rceil)} \rightarrow R$ such that $\psi$ can be represented by ${1 \over s} \phi$ where $s \in T$.  Fix $a \in T^{-1} \ba_{p^e-1}$ and $x \in T^{-1} I$ and consider $\psi(a x)$.

Note that $x$ can be represented as $z/t^{p^{e'}}$ where $z \in I$, $t \in T$ and $e' = ne$ for some $n$.   Likewise, $a$ can be represented as $w/u^{p^{e''}}$ where $w \in \ba_{p^{e} - 1}$, $u \in T$ and $e'' = me$ for some $m$.  Then $\psi(a x) = \psi(w/u^{p^{e''}} z/t^{p^{e'}}) = {1 \over u^m t^n} \psi(w z)$.  But we know that $\psi(w z)$ is identified with ${1 \over s} \phi(w z)$ and $\phi(w z) \in I$ by hypothesis. This implies that $\psi(ax) \in T^{-1} I$ as desired.

\end{proof}

We also have the following partial converse to the above statement; again compare with \cite[Proposition 5.3]{LyubeznikSmithCommutationOfTestIdealWithLocalization}.

\begin{lemma}\cite[Proposition 5.3]{LyubeznikSmithCommutationOfTestIdealWithLocalization}
\label{UniformlyFCompatibleIdealsUnlocalize}
Suppose that $R$ is $F$-finite, $(R, \Delta, \ba_{\bullet})$ is a triple and $T$ is a multiplicative system on $R$.  Let $I$ be an ideal of $T^{-1} R$ and let us denote by $f$ the natural map $f : R \rightarrow S = {T^{-1} R}$.  If $I$ is uniformly $(\Delta |_{\Spec S}, (\ba S) _{\bullet}, F)$-compatible, then $f^{-1}(I)$ is uniformly $(\Delta, \ba_{\bullet}, F)$-compatible.
\end{lemma}
\begin{proof}
Suppose not; that is, suppose that $f^{-1}(I)$ is not uniformly $(\Delta, \ba_{\bullet}, F)$-compatible.  Therefore let
\[
\psi : \Fr{e}{R\lceil (p^e - 1) \Delta \rceil)} \rightarrow R
\]
be a map that sends some element $x \in \ba_{p^e - 1} f^{-1}(I) \subseteq R \subseteq R(\lceil (p^e - 1) \Delta \rceil)$ to some element $z \notin f^{-1}(I)$. But then, tensoring the map with $\tensor_R S$ contradicts our hypothesis.
\end{proof}

The same proof also gives us the following statement substituting completion for localization.

\begin{lemma}
\label{LemmaUniformlyCompatibleIdealsUncomplete}
Suppose that $R$ is an $F$-finite local ring and suppose $(R, \Delta, \ba_{\bullet})$ is a triple.  Let us denote by $f$ the natural map $f : R \rightarrow \hat R$.  If $I \subseteq \hat{R}$ is uniformly $(\hat{\Delta}, \hat \ba_{\bullet}, F)$-compatible then $f^{-1}(I)$ is uniformly $(\Delta, \ba_{\bullet}, F)$-compatible.
\end{lemma}

\begin{lemma}
\label{LemmaUniformlyCompatibleIdealsComplete}
Suppose that $R$ is an $F$-finite local ring and suppose that $(R, \Delta, \ba_{\bullet})$ is a triple.  Let $I$ be an ideal of $R$ and let us denote by $f$ the natural map $f : R \rightarrow \hat R$.  If $I \subseteq R$ is uniformly $(\Delta, \ba_{\bullet}, F)$-compatible, then $I \hat{R}$ is uniformly $(\hat{\Delta}, \hat \ba_{\bullet}, F)$-compatible.
\end{lemma}
\begin{proof}
Since $R$ is $F$-finite, notice that $\Hom_R(\Fr{e}{R(\lceil (p^e -1 )\Delta \rceil)}, R) \tensor_R \hat{R}$ is isomorphic to  $\Hom_{\hat R}(\Fr{e}{\hat{R}(\lceil (p^e - 1)\hat{\Delta}\rceil), \hat{R})}$.  Let $\{x_1, \ldots, x_n\}$ be generators for $\Fr{e}{I}$ as an $R$-module.  This implies that their images are generators for $\Fr{e}{\hat{I}}$ as an $\hat{R}$-module.  But then, it is enough to check that for every $e > 0$,  every $a \in \hat{\ba}_{p^e - 1}$, and every homomorphism $\phi \in \Hom_{\hat R}(\Fr{e}{\hat{R}(\lceil (p^e - 1)\hat{\Delta}\rceil, \hat{R})}$, we have $(\phi \circ \Fr{e}{(\times a)}) (f(x_i)) \subseteq \hat{I}$.  However, again, we may generate $\Frp{e}{\hat{\ba}_{p^e - 1}}$ as an $\hat{R}$-module, by the images of elements of $\ba_{p^e - 1}$, and so we may assume that the $a$ above may be taken to be one of these elements.  But then it becomes clear that $\phi \circ \Fr{e}{(\times a)} (f(x_i)) \subseteq \hat{I}$ since the result holds for $\phi$ coming from $\Hom_R(\Fr{e}{R(\lceil (p^e -1 )\Delta \rceil)}, R)$.
\end{proof}

We can also link uniformly $F$-compatible ideals with Fedder-type criteria (in fact we use Fedder's original machinery).  We first recall a lemma of Fedder (which was stated in slightly more general context originally).  Also compare with \cite[Proposition 5.2]{LyubeznikSmithCommutationOfTestIdealWithLocalization}.

\begin{lemma}
\label{LemmaFeddersOriginal}
\cite[Lemma 1.6]{FedderFPureRat}
Suppose that $S$ is an $F$-finite regular local ring and $R = S/I$.  Then,
\begin{itemize}
\item[(a)]  $\Hom_S(\Fr{e}{S}, S) \cong \Fr{e}{S}$ as an $\Fr{e}{S}$-module.
\item[(b)]  Let $T$ be a generator for $\Hom_S(\Fr{e}{S}, S)$ as an $\Fr{e}{S}$-module, and $J$ an ideal.  Let $x$ be an element of $S$.  Then the image of $\Fr{e}{J} \subset \Fr{e}{S}$ under the homomorphism $T \circ \Fr{e}{(\times x)}: \Fr{e}{S} \rightarrow S$ is contained in $J$ if and only if $x \in (J^{[p^e]} : J)$.
\item[(c)]  There exists an isomorphism $\phi: \Fr{e}{((I^{[p^e]} : I)/I^{[p^e]})} \rightarrow \Hom_R(\Fr{e}{R}, R)$ defined by $\phi(\overline{s}) = \overline{T \circ \Fr{e}{(\times s)}}$ where $T$ is any choice of an $\Fr{e}{S}$-module generator for the module $\Hom_S(\Fr{e}{S}, S)$.
\end{itemize}
 \end{lemma}

\begin{proposition}
\label{PropFedderCriterionForCompatible}
Suppose that $S$ is an $F$-finite regular ring and that $R = S/I$.  Further suppose that $\ba_{\bullet}$ is a graded system of ideals such that $(S, \ba_{\bullet})$ and $(R, \overline{\ba_{\bullet}})$ are pairs.  Let $J$ be another ideal containing $I$.  Then $\overline{J} \subseteq R$ is uniformly $(\overline{\ba_{\bullet}}, F)$-compatible if and only if for all $e \geq 0$ and all $a \in \ba_{p^e-1}$ we have
\[
a (I^{[p^e]} : I) \subseteq (J^{[p^e]} : J).
\]
In other words, $\overline{J} \subseteq R$ is uniformly $(\overline{\ba_{\bullet}}, F)$-compatible if and only if for all $e \geq 0$ we have
\[
\ba_{p^e-1} (I^{[p^e]} : I) \subseteq (J^{[p^e]} : J).
\]
\end{proposition}
\begin{proof}
Without loss of generality, we may assume that $S$ is local.
Fix $T$ as in Lemma \ref{LemmaFeddersOriginal}.
Now suppose that for all $e \geq 0$, all $a \in \ba_{p^e-1}$ we have $a (I^{[p^e]} : I) \subseteq (J^{[p^e]} : J)$.
Note that, for every composition $\xymatrix{\phi : \Fr{e}{R} \ar[r]^-{\Fr{e}{(\times a)}} & \Fr{e}{R} \ar[r]^-{\alpha} & R}$, we can represent $\alpha$ as $T \circ \Fr{e}{(\times s)}$ for some $s \in (I^{[p^e]} : I)$.  But then $\phi$ can be represented by $T \circ \Fr{e}{(\times a s)}$.  But $a s \in (J^{[p^e]} : J)$, which implies that $\phi$ must send $\Fr{e}{(\overline{J})} \subseteq \Fr{e}{R}$ into $\overline{J}$, and so $\overline{J}$ is uniformly $(\overline{\ba}_{\bullet}, F)$-compatible.

Conversely, suppose that $\overline{J}$ is uniformly $(\ba_{\bullet}, F)$-compatible.  This means that for all $e \geq 0$ and all $a \in \ba_{(p^e-1)}$ and all compositions
\[
\xymatrix{\phi : \Fr{e}{R} \ar[r]^-{\Fr{e}{(\times a)}} & \Fr{e}{R} \ar[r]^-{\alpha} & R}
\]
we have $\phi(\Fr{e}{\overline{J}}) \subseteq \overline{J}$.  Pick arbitrary elements $s \in (I^{[p^e]} : I)$ and $a \in \ba_{(p^e-1)}$, and then consider the corresponding $\phi$ which is represented by $T \circ \Fr{e}{(\times a s)}$.  But note that $\phi(\Fr{e}{\overline{J}}) \subseteq \overline{J}$, which implies that $T \circ \Fr{e}{(\times a s)}(\Fr{e}{J}) \subseteq J$.  But then, by Lemma \ref{LemmaFeddersOriginal}, we see that $a s \in (J^{[p^e]} : J)$.  This implies that $a (I^{[p^e]} : I) \subseteq (J^{[p^e]} : J)$.
\end{proof}

\begin{remark}
There doesn't seem to be a way to formulate Proposition \ref{PropFedderCriterionForCompatible} in a way that involves divisors $\Delta$ on $R$.  Note that, since every integral divisor on $\Spec S$ is Cartier, any such divisor's contribution can be absorbed into the term $\ba_{\bullet}$.  On the other hand, some divisors on $\Spec R$ may not be the restriction of divisors from $\Spec S$.
\end{remark}

\begin{remark}
\label{RemarkLinksWithFESubmodules}
In particular, by \cite[Proposition 5.2]{LyubeznikSmithCommutationOfTestIdealWithLocalization} we see that uniformly $F$-compatible ideals can be characterized by the fact that they annihilate $\mathcal{F}(E)$-submodules of $E$ (where $E$ is the injective hull of the residue field $R/\bm$).  Also compare with Lemma \ref{LemmaFSubmoduleLocalCharacterizationOfCenterOfFPurity}.  Note that in \cite{LyubeznikSmithCommutationOfTestIdealWithLocalization}, the authors consider submodules of $E$ that are stable under the additive maps $\phi_e : E \rightarrow E$ which satisfy $\phi_e(rm) = r^{p^e} \phi_e(m)$.  We instead consider submodules (ideals) of $R$ that are stable under the additive maps $\psi_e : R \rightarrow R$ which satisfy $\psi_e(r^{p^e} m) = r \psi_e(m)$.  \end{remark}

\begin{corollary}
\label{CorFinitisticTestIdealIsCompatible}
If $R$ is a quotient of an $F$-finite regular local ring, then the finitistic test ideal $\tau(R, \ba_{\bullet})$ is uniformly $F$-compatible.
\end{corollary}
\begin{proof}
The proof is the same as in \cite{SchwedeSharpTestElements} (respectively, \cite{VassilevTestIdeals}) where the same result was shown for $\tau(R, \ba^t)$ (respectively, for $\tau(R)$).
\end{proof}

\section{Centers of $F$-purity}
\label{SectionCentersOfFPurity}

In this section, we give the definition of centers of $F$-purity as inspired by the remarks in the introduction.  We then link this notion to uniformly $F$-compatible ideals.

\begin{definition}
Suppose $R$ is an $F$-finite reduced ring and that $(R, \Delta, \ba_{\bullet})$ is a triple as defined in Subsection \ref{SubsectionPositiveCharacteristicTriples}.  We say that a (possibly non-closed) point $P \in X = \Spec R$ is a \emph{center of sharp $F$-purity for $(R, \Delta, \ba_{\bullet})$} if, for every $f \in P R_P$, every $e > 0$ and every $a \in \ba_{p^e-1}$, the composition map
\[
\xymatrix{
R_P \ar[r] & \Fr{e}{R_P} \ar[r]^-{\Fr{e}{(\times f a)}} & \Fr{e}{R_P} \ar[r] & \Fr{e}{R_P(\lceil (p^e - 1) \Delta|_{\Spec R_P} \rceil)}
}
\]
does not split.  If $\Delta = 0$ and $\ba_i = R$ for all $i \geq 0$, then we simply say that such a $P \in X$ is a \emph{center of $F$-purity}.
\end{definition}

\begin{remark}
As we will see, these notions have the most interesting implications under the hypothesis that the triple $(R, \Delta, \ba_{\bullet})$ is sharply $F$-pure.
\end{remark}

\begin{remark}
In the previous definition, it is equivalent to consider all $a \in (\ba_{p^e - 1} R_P)$ or all $f \in P$.  Likewise, it is equivalent to pick $f \in P$ instead of $f \in P R_P$.
\end{remark} 

\begin{remark}
 \label{RemarkRelationToSplittingPrime}
Suppose that $(R, \bm)$ is an $F$-finite $F$-pure reduced local ring.  Aberbach and Enescu previously defined the \emph{splitting prime} of $R$, denoted $\mathcal{P} = \mathcal{P}(R)$, to be the set of elements $c \in R$ such that the map $R \rightarrow \Fr{e}{R}$ which sends $1$ to $c$, does not split; see \cite{AberbachEnescuStructureOfFPure}.  Furthermore, they proved that $\mathcal{P}$ is a prime ideal and is compatible with localization at primes containing $\mathcal{P}$, see \cite[Theorem 3.3, Proposition 3.6]{AberbachEnescuStructureOfFPure}.  Therefore, by localizing at $\mathcal{P}$ itself, we immediately see that $\mathcal{P}$ is a center of $F$-purity.  On the other hand, suppose there was a center of $F$-purity $Q$ with $\mathcal{P} \subsetneq Q$.  Then for $c \in Q \setminus \mathcal{P}$, the map $R \rightarrow \Fr{e}{R}$ which sends $1$ to $c$ cannot split (since it doesn't split after localizing at $Q$).  But that is impossible since it would imply that $c \in \mathcal{P}$.  Therefore we see that $\mathcal{P}$ is the unique largest center of $F$-purity in $(R, \bm)$.
\end{remark}


Centers of  sharp $F$-purity are analogues of centers of log canonicity in characteristic zero.  The following simple observation gives a certain amount of evidence for this (we will prove a generalization of this result in Proposition \ref{PropTestIdealIsCompatible}).

\begin{proposition}
Suppose that $R$ is $F$-finite and reduced and that $(R, \Delta, \ba_{\bullet})$  is a triple.  Then the minimal primes of the non-strongly $F$-regular locus of the triple are centers of sharp $F$-purity.
\end{proposition}
\begin{proof}
Suppose $P$ is such a minimal prime.  Since being strongly $F$-regular is a local condition, we see that $(R_P, \Delta|_{\Spec R_P}, (\ba R_P)_{\bullet})$ is strongly $F$-regular on the punctured spectrum (that is, except at the unique closed point $P$) but is \emph{not} strongly $F$-regular at $P$.  If $P$ is not a center of sharp $F$-purity, then there exists an $f \in P \O_{X, P}$ and an $a \in \ba_{p^e-1}$ such that the composition
\[
\xymatrix{
R_P \ar[r] & \Fr{e}{R_P} \ar[r]^-{\Fr{e}{(\times f a)}} & \Fr{e}{R_P} \ar[r] & \Fr{e}{R_P(\lceil (p^e - 1) \Delta|_{\Spec R_P} \rceil)}
}
\]
splits.  But this is impossible by Lemma \ref{LemmaSplitAlongAndComplementImpliesFRegular}.
\end{proof}

\begin{corollary}
Suppose $R$ is an $F$-finite domain and suppose that $(R, \Delta, \ba_{\bullet})$ is a triple.  Then $(R, \Delta, \ba_{\bullet})$ is strongly $F$-regular if and only if $(R, \Delta, \ba_{\bullet})$ has no proper non-trivial centers of sharp $F$-purity.
\end{corollary}

We will prove later that the centers of log canonicity of a pair $(X, \Delta)$ of dense sharply $F$-pure type, reduce to centers  of $F$-purity for all but finitely many primes, see Theorem \ref{TheoremCentersOfLogCanonicityReduce}.

\begin{proposition}
Centers of sharp $F$-purity for $(R, \Delta, \ba_{\bullet})$ are precisely the prime ideals that are uniformly $(\Delta, \ba_{\bullet}, F)$-compatible.
\end{proposition}
\begin{proof}
Suppose that $P$ is not a center of sharp $F$-purity.  Then for some $f \in P R_P$, some $e > 0$ and some $a \in \ba_{p^e - 1}$, the composition
\[
\xymatrix{
R_P \ar[r] & \Fr{e}{R_P} \ar[r]^-{\Fr{e}{(\times f a)}} & \Fr{e}{R_P} \ar[r] & \Fr{e}{R_P(\lceil (p^e - 1)\Delta|_{\Spec R_P} \rceil)}
}
\]
splits.  The splitting map sends an element of $a P R_P$ to $1$, and $1$ is not in $P R_P$.  This proves that $P$ cannot be uniformly $(\Delta, \ba_{\bullet}, F)$-compatible by Lemma \ref{UniformlyFCompatibleIdealsLocalize}.  The converse is similar.
\end{proof}

\begin{corollary}
\label{PropPrimaryDecompositionOfUniformlyFCompatibleIdeals}
Suppose that $(R, \Delta, \ba_{\bullet})$ is a triple and that $I$ is a radical ideal that is uniformly $(R, \Delta, \ba_{\bullet})$-compatible.  Then the associated primes of $I$ are centers of $F$-purity.  In particular, if $(R, \Delta, \ba_{\bullet})$ is sharply $F$-pure, then every associated prime of every uniformly $(\Delta, \ba_{\bullet}, F)$-compatible ideal is a center of sharp $F$-purity for the triple $(R, \Delta, \ba_{\bullet})$.
\end{corollary}
\begin{proof}
Simply localize at the prime in question and then use Lemma \ref{UniformlyFCompatibleIdealsUnlocalize}.
\end{proof}



\begin{corollary}
Suppose that $R$ is $F$-pure and that $P$ is a center of $F$-purity, consider $Z = V(P)$ as a subscheme of $X = \Spec R$.  Then $X$ is compatibly $F$-split along $Z$.
\end{corollary}
\begin{proof}
Let $\phi : \Fr{e}{\O_X} \rightarrow \O_X$ be an $\O_X$-module homomorphism that sends $1$ to $1$.  By assumption, $\phi(\Fr{e}{P}) \subseteq P$.  Note that in this case we have equality, $\phi(\Fr{e}{P}) =  P$, since $\phi(x^p) = x$.
\end{proof}

Finally, again using this notion of centers of $F$-purity, we see that the radical of a uniformly $F$-compatible ideal is uniformly $F$-compatible.

\begin{proposition}
Suppose that $(R, \Delta, \ba_{\bullet})$ is a triple and suppose that $I$ is a uniformly $(\Delta, \ba_{\bullet}, F)$-compatible ideal.  Then $\sqrt{I}$ is also $(\Delta, \ba_{\bullet}, F)$-compatible.
\end{proposition}
\begin{proof}
It will be enough to show that the minimal associated primes of $I$ are uniformly $(\Delta, \ba_{\bullet}, F)$-compatible, since we can then use Lemma \ref{EasyPropertiesOfUniformlyFCompatibleIdeals}.  Therefore, let $Q$ be a prime corresponding to a minimal primary component of the primary decomposition of $I$.  Notice that by assumption $\sqrt {I R_Q} = Q R_Q$.  There are now two possibilities:
\begin{itemize}
\item[(i)]  The first is that $I R_Q = Q R_Q$, in which case $Q$ is uniformly $(\Delta, \ba_{\bullet}, F)$-compatible by Lemmas \ref{UniformlyFCompatibleIdealsLocalize} and \ref{UniformlyFCompatibleIdealsUnlocalize}.
\item[(ii)]  If $I R_Q \neq Q R_Q$, then we see that the triple $(R_Q, \Delta|_{\Spec R_{Q}}, (\ba R_Q)_{\bullet})$ is not sharply $F$-pure by Corollary \ref{CorUniformlyFCompatibleDefinesFPureSubscheme}, which implies that $Q$ is a center of sharp $F$-purity for the triple $(R, \Delta, \ba_{\bullet})$.  But then $Q$ is uniformly $(\Delta, \ba_{\bullet}, F)$-compatible.
\end{itemize}
\end{proof}

Using the technique of Fedder's criterion, we can also easily characterize centers of sharp $F$-purity for pairs $(R, \ba_{\bullet})$, see Proposition \ref{PropFedderCriterionForCompatible}.

\section{Relations to $F$-stable submodules of $H^d_{\bm}(R)$ and finiteness of centers of sharp $F$-purity}
\label{SectionRelationsToFStableSubmodulesAndFinitenessOfCenters}

In this section we will show that there are only finitely many centers of $F$-purity in the case of an $F$-pure local ring, and give an alternate explanation of why the notions we have been considering are very closely related to the theory of $F$-stable submodules of $H^d_{\bm}(R)$ (as studied, for example, in \cite{SmithFRatImpliesRat}).  In fact, it is by applying the machinery that was used to study the $F$-stable submodules of $H^d_{\bm}(R)$ by Enescu-Hochster and Sharp that we are able to prove that there are only finitely many centers of $F$-purity for an $F$-pure pair.

\begin{lemma}
\label{LemmaFSubmoduleLocalCharacterizationOfCenterOfFPurity}
Suppose that $R$ is an $F$-finite ring and that $(R, \Delta, \ba_{\bullet})$ is a triple.  Fix an ideal $I \subseteq R$.  The following are equivalent:
\begin{itemize}
\item[(a)]  $I$ is uniformly $(\Delta, \ba_{\bullet}, F)$-compatible.
\item[(b)]  For every $e > 0$ and every $a \in \ba_{p^e -1}$ and $f \in I$, the composition
\[
\small
\xymatrix@C=9pt{
\Hom_{R}(\Fr{e}{R(\lceil (p^e - 1) \Delta \rceil)}, R) \ar[r] & \Hom_R(\Fr{e}{R}, R) \ar[rr]^-{\Fr{e}{(\times a f)}} & & \Hom_{R}(\Fr{e}{R}, R) \ar[r] & \Hom_{R}(R, R) = R \ar[r] & R/I
}
\]
is zero.
Here the first three maps in the composition are simply $\Hom_R(\blank, R)$ applied to
\[
\xymatrix@C=35pt{R \ar[r] & \Fr{e}{R} \ar[r]^-{\Fr{e}{(\times a f)}} & \Fr{e}{R} \ar[r] & \Fr{e}{R(\lceil (p^e - 1) \Delta \rceil)}}.
\]

\vskip 10pt
\item[(*)] If, in addition, we assume that $R$ is a local ring with maximal ideal $\bm$ and we use $E_R$ to denote the injective hull of the residue field $R/\bm$, then condition (c) below is also equivalent to (a) and (b).
\item[(c)]   For every $e > 0$ and every $a \in \ba_{p^e -1}$ and every $f \in I$, the composition
\[
\xymatrix{
E_{R/I} \ar[r] & E_R \ar[r] & E_R \tensor_R \Fr{e}{R} \ar[rr]^-{\Fr{e}{(\times a f)}} & & E_R \tensor_R \Fr{e}{R} \ar[r] & E_R \tensor_R \Fr{e}{R(\lceil (p^e - 1) \Delta \rceil)}
}
\]
is zero.
\end{itemize}
\end{lemma}
\begin{proof}
By the usual application of Matlis duality, it is clear that conditions (b) and (c) are equivalent; see for example \cite[Lemma 3.4]{TakagiInversion}.  On the other hand, $I$ is uniformly $(\Delta, \ba_{\bullet}, F)$-compatible if and only if for every map $\phi : \Fr{e}{R(\lceil (p^e - 1) \Delta \rceil)} \rightarrow R$ and every $a \in \ba_{p^e -1}$ we have $\phi(\Fr{e}{aI}) \subseteq I$.  But this implies that for any $f \in I$, the image of any composition
\[
\xymatrix{R \ar[r] & \Fr{e}{R} \ar[rr]^-{\Fr{e}{(\times a f)}} & & \Fr{e}{R} \ar[r] &  \Fr{e}{R(\lceil (p^e - 1) \Delta \rceil)} \ar[r]^-{\phi} & R}
\]
is contained in $I$, which certainly implies that the map in (b) is zero.  Conversely, if (a) is false, then there exists some $a \in \ba_{p^e -1}$, $f \in I$ and $\phi : \Fr{e}{R(\lceil (p^e - 1) \Delta \rceil)} \rightarrow R$ such that $\phi(a f) \notin I$.  But then the associated composition from (b) is non-zero.
\end{proof}

\begin{remark}
Suppose that $(R, \bm)$ is a quasi-Gorenstein local, $\Delta = 0$, $\ba_i = R$ for all $i \geq 0$ and $I \subset R$ is an ideal.  Also set $N = \Ann_{E_R} I \subset E_R \cong H^{d}_{\bm}(R)$.  Then we see that $I$ is uniformly $F$-compatible if and only if $N$ is a $F$-stable submodule of $H^{d}_{\bm}(R)$.
\end{remark}

This implies that the question of whether there are finitely many centers of $F$-purity is closely related to the question of whether there are finitely many $F$-stable submodules of $H^d_{\bm}(R)$.  Thus we obtain the following corollary:

\begin{corollary}
Let $R$ be a reduced local ring.  Suppose further that $(R, \Delta, \ba_{\bullet})$ is sharply $F$-pure.  Then there are only finitely many uniformly $(\Delta, \ba_{\bullet}, F)$-compatible ideals $I$.  In particular, there are only finitely many centers of $F$-purity for $(R, \Delta, \ba_{\bullet})$.
\end{corollary}
\begin{proof}
We have observed previously that uniformly $(\Delta, \ba_{\bullet}, F)$-compatible ideals are reduced, see Corollary \ref{CorUniformlyFCompatibleDefinesFPureSubscheme}, and also that they are closed under sum and intersection, see Lemma \ref{EasyPropertiesOfUniformlyFCompatibleIdeals}.  Proposition \ref{PropPrimaryDecompositionOfUniformlyFCompatibleIdeals} implies that the set of uniformly $(\Delta, \ba_{\bullet}, F)$-compatible ideals are closed under taking primary decomposition.  But then it follows that there can be at most finitely many such ideals by \cite[Theorem 3.1]{EnescuHochsterTheFrobeniusStructureOfLocalCohomology}.  One can obtain an alternate proof using the techniques of \cite{SharpGradedAnnihilatorsOfModulesOverTheFrobeniusSkewPolynomialRing}.
\end{proof}



\begin{remark}
\label{RemarkFIdealsAreCompatible}
Suppose that $(R, \bm)$ is an $F$-finite local ring and suppose that $N \subseteq H^d_{\bm}(R)$ is an $F$-stable submodule of $H^d_{\bm}(R)$.  The proof of \cite[Theorem 4.1]{EnescuHochsterTheFrobeniusStructureOfLocalCohomology} shows that if $I = \Ann_R N$, then $I$ is uniformly $F$-compatible (they only state that $I$ is ``compatible'' with splittings, but the proof of the general case is the same).  In fact, the same proof can be used to show the following result (also compare with the proof of Theorem \ref{TheoremBigVsNonFinitisticTestIdeals}).
\end{remark}

\begin{proposition}
Suppose that $M$ is an $R$-module and $(R, \Delta, \ba_{\bullet})$ is a triple.   For all $e \geq 0$ and all $a \in \ba_{p^e - 1}$, use $\phi_{e, a} : R \rightarrow \Fr{e}{R(\lceil (p^e - 1)\Delta \rceil)}$ to denote the map which sends $1$ to $a$.  Suppose that $N \subseteq M$ and that $J = \Ann_R N$.  Finally suppose that the image of $N$ under the map
\[
\gamma_{e, a} := M \tensor \phi_{e, a} : M \rightarrow M \tensor_R \Fr{e}{R(\lceil (p^e - 1)\Delta \rceil)}
\]
is annihilated by $\Fr{e}{J}$ for all $e$ and $a$.  Then $J$ is uniformly $(\Delta, \ba_{\bullet}, F)$-compatible.
\end{proposition}
The argument of the proof is essentially the same as in \cite[Theorem 4.1]{EnescuHochsterTheFrobeniusStructureOfLocalCohomology}, we simply generalize it to the context of triples $(R, \Delta, \ba_{\bullet})$.  We provide it for the convenience of the reader.
\begin{proof}
Consider a map $\psi : \Fr{e}{R(\lceil (p^e - 1)\Delta \rceil)} \rightarrow R$ and note that $\psi$ induces a map
\[ \delta : M \tensor_R \Fr{e}{R(\lceil (p^e - 1)\Delta \rceil)} \rightarrow M \tensor_R R \cong M \]
We need to show that $\psi(\Fr{e}{(\ba_{p^e - 1} J)}) \subseteq J = \Ann_R(N)$. Therefore choose $n \in N$, $a \in \ba_{p^e - 1}$ and $x \in \Fr{e}{J} \subset \Fr{e}{R} \subseteq \Fr{e}{R(\lceil (p^e - 1)\Delta \rceil)}$.  We note that $\delta(n \tensor a x) = n \tensor \psi(a x) = \psi(a x) n \in M \cong M \tensor R$.  But $n \tensor ax = (n \tensor a).x = 0 \in M \tensor_R \Fr{e}{R(\lceil (p^e - 1)\Delta \rceil)}$ by hypothesis.
\end{proof}

\section{Relations to big test ideals and multiplier ideals}
\label{SectionRelationToBigTestIdealsAndMultiplierIdeals}

In this section, we show that big (non-finitistic) test ideals are uniformly $F$-compatible and that the big test ideal is the smallest uniformly $F$-compatible ideal whose intersection with $R^{\circ}$ is non-trivial.  Note that the usual (finitistic) test ideal is uniformly $F$-compatible by \cite{VassilevTestIdeals} or, in the case of pairs $(R, \ba_{\bullet})$, by Corollary \ref{CorFinitisticTestIdealIsCompatible} (compare with \cite{SchwedeSharpTestElements}.  See Section \ref{SectionAppendixOnBigTestIdeals} for basic definitions of big (non-finitistic) test ideals, or see \cite{HaraYoshidaGeneralizationOfTightClosure} or \cite{LyubeznikSmithCommutationOfTestIdealWithLocalization} for an alternate point of view.

\begin{proposition}  \cite[Theorem 6.2]{LyubeznikSmithCommutationOfTestIdealWithLocalization}, \cite{SmithDModuleStructure} \label{PropTestIdealIsCompatible}
Given a triple $(R, \Delta, \ba_{\bullet})$, the ideal $\tau_b(\Delta, \ba_{\bullet})$ is uniformly $(\Delta, \ba_{\bullet}, F)$-compatible.
\end{proposition}
\begin{proof}
By Lemmas \ref{UniformlyFCompatibleIdealsLocalize} and \ref{UniformlyFCompatibleIdealsUnlocalize} and the fact that the big test ideal behaves well with respect to localization, we see that it is harmless to assume that $R$ is a local ring.  Likewise, by Lemmas \ref{LemmaUniformlyCompatibleIdealsUncomplete} and \ref{LemmaUniformlyCompatibleIdealsComplete} and since the big test ideal behaves well with respect to completion, see \cite{HaraTakagiOnAGeneralizationOfTestIdeals}, we see that it is harmless to assume that $R$ is complete.

We need to show that the composition
\[
\xymatrix@C=15pt{
E_{R/\tau_b(\Delta, \ba_{\bullet})} \ar[r] & E_R \ar[r] & E_R \tensor_R \Fr{e}{R} \ar[rr]^-{F^e_*{(\times a f)}} & & E_R \tensor_R \Fr{e}{R} \ar[r] & E_R \tensor_R \Fr{e}{R(\lceil (p^e - 1) \Delta \rceil)}
}
\]
is zero for every $e \geq 0$ and every $f \in \tau_b(\Delta, \ba_{\bullet})$.  But this follows immediately from the fact that $\tau_b(\Delta, \ba_{\bullet})$ is made up of big sharp test elements for $(R, \Delta, \ba_{\bullet})$ and that $0^{*\Delta, \ba_{\bullet}}_{E_R} = E_{R/\tau_b(\Delta, \ba_{\bullet})}$.
\end{proof}

\begin{remark}
It is natural to ask whether the finitistic test ideal of a triple $(R, \Delta, \ba_{\bullet})$ is uniformly $F$-compatible in the case that $\Delta \neq 0$.  If one defines $\tau(R, \Delta, \ba_{\bullet})$ to be the ideal made up of ``test elements for finite modules'' for $(R, \Delta, \ba_{\bullet})$, then one can use an argument similar to the one in Theorem \ref{TheoremBigVsNonFinitisticTestIdeals} to show that $\tau(R, \Delta, \ba_{\bullet})$ is uniformly $F$-compatible.  However, if one defines $\tau(R, \Delta, \ba_{\bullet})$ to be $\Ann_R 0^{* \Delta, \ba_{\bullet}, fg}$ then I do not see an easy approach to the result.  On the other hand, if $K_R + \Delta$ is $\bQ$-Cartier then one can use the technique of \cite[Theorem 2.8]{TakagiInterpretationOfMultiplierIdeals}.
\end{remark}

\begin{theorem}
\label{TheoremTestIdealIsSmallestFCompatible}
Given a triple $(R, \Delta, \ba_{\bullet})$, the ideal $\tau_b(\Delta, \ba_{\bullet})$ is the smallest ideal which is uniformly $(\Delta, \ba_{\bullet}, F)$-compatible and whose intersection with $R^{\circ}$ is non-empty.
\end{theorem}
\begin{proof}
As above, we may assume that $R$ is local.
If the conclusion of the theorem is false then by Lemma \ref{EasyPropertiesOfUniformlyFCompatibleIdeals} we may assume that there exists some $I \subsetneq \tau_b(\Delta, \ba_{\bullet})$ which is uniformly $(\Delta, \ba_{\bullet}, F)$-compatible and which satisfies $I \cap R^{\circ} \neq \emptyset$.  This implies that for every $e \geq 0$, every $a \in \ba_{p^e - 1}$ and every $f \in I \cap R^{\circ}$, the composition
\[
\xymatrix{
E_{R/I} \ar[r] & E_R \ar[r] & E_R \tensor_R \Fr{e}{R} \ar[rr]^-{\Fr{e}{(\times a f)}} & & E_R \tensor_R \Fr{e}{R} \ar[r] & E_R \tensor_R \Fr{e}{R(\lceil (p^e - 1) \Delta \rceil)}
}
\]
is zero.  But this implies that $E_{R/I}$ is contained in the $(\Delta, \ba_{\bullet})$-tight closure of zero in $E_R$.  However, $I = \Ann_R E_{R/I} \supseteq \Ann_R 0^{* \Delta, \ba_{\bullet}}_{E_R} = \tau_b(\Delta, \ba_{\bullet})$, which is impossible.
\end{proof}

\begin{corollary}
\label{CorollaryFedderCriterionForTestIdeals}
Suppose that $R$ is a quotient of a regular local ring $S$ by an ideal $I$ and denote the projection $S \rightarrow S/I = R$ by $\pi$. Further suppose that $\ba_\bullet$ is a graded system of ideals of $S$ such that both $(S, \ba_{\bullet})$ and $(R, \overline{\ba}_{\bullet})$ are pairs.  Consider the ideals $J$ that contain $I$ and that $(J/I) \cap R^{\circ} \neq \emptyset$.  Then $\pi^{-1}\tau_b(R,\ba_{\bullet})$ is the unique smallest ideal of this set which also satisfies the property that
\[
\ba_{p^e - 1} (I^{[p^e]} : I) \subseteq (J^{[p^e]} : J).
\]
for all $e > 0$.
\end{corollary}

\begin{remark}
We note that in the case that $R = S$, Corollary \ref{CorollaryFedderCriterionForTestIdeals} looks very similar to the characterization of test ideals given in \cite[Definition 2.9]{BlickleMustataSmithDiscretenessAndRationalityOfFThresholds}.  There, in the context of a pair $(R, \ba^t)$ where $R$ is regular, they show that the test ideal is the smallest ideal $J$ such that $\ba^{\lceil tp^e \rceil} \subseteq J^{[p^e]}$ for all $e > 0$.  On the other hand, in the case that $\ba_i = R$ for all $i \geq 0$ (but $R$ is possibly singular), the characterization of the big test ideal in Corollary \ref{CorollaryFedderCriterionForTestIdeals} coincides precisely with a criterion given in \cite{LyubeznikSmithCommutationOfTestIdealWithLocalization}.
\end{remark}

In the case that $R = S$, this gives an alternate proof of the following result of Takagi; see \cite[Proof of Theorem 2.4]{TakagiFormulasForMultiplierIdeals}.

\begin{corollary}
Suppose that $R$ is a regular local ring and $(R, \ba_{\bullet})$ is a pair.  Then $\tau_b(\ba_{\bullet})^{* \ba_{\bullet}} = R$.
\end{corollary}
\begin{proof}
Use Corollary \ref{CorollaryFedderCriterionForTestIdeals} and observe that $\tau_b(R, \ba_{\bullet}) \ba_{p^e - 1} 1^{p^e} \subseteq \tau_b(R, \ba_{\bullet})^{[p^e]}$ for all $e \gg 0$.
\end{proof}

Since test ideals are closely related to multiplier ideals (see \cite{SmithMultiplierTestIdeals}, \cite{HaraInterpretation}, \cite{HaraYoshidaGeneralizationOfTightClosure}, \cite{TakagiInterpretationOfMultiplierIdeals}) it is natural to ask whether multiplier ideals (or even multiplier ideal-like constructions such as adjoint ideals, see \cite{LazarsfeldPositivity2} and \cite{TakagiPLTAdjoint}) always yield uniformly $(\Delta, F)$-compatible ideals.  As we will see, the answer to this is affirmative.

\begin{theorem}
\label{TheoremGeneralizationOfCentersOfLogCanonicityAreFCompatible}
Suppose $R$ is a normal domain and a quotient of a regular ring essentially of finite type over a perfect field.  Set $X = \Spec R$ and suppose that $(X, \Delta)$ is a pair such that $K_X + \Delta$ is $\bQ$-Cartier.  Further suppose that $\pi : \tld X \rightarrow X$ is a proper birational morphism and that $\tld X$ is normal.  If $G$ is any effective integral divisor on $\tld X$ such that $\pi_* \O_{\tld X}(\lceil K_{\tld X} - \pi^*(K_X + \Delta) \rceil + G)$ is naturally a submodule of $\O_X$ (in particular, this happens if $G$ is exceptional), then $\pi_* \O_{\tld X}(\lceil K_{\tld X} - \pi^*(K_X + \Delta) \rceil + G)$ is uniformly $(\Delta, F)$-compatible.
\end{theorem}
The proof uses the same technique as the proof of the main theorem of \cite{HaraWatanabeFRegFPure}.
\begin{proof}
First note, that we may assume that $\pi_* K_{\tld X} = K_X$ (as divisors) and that $K_{\tld X}$ (and thus $K_X$) is an integral Weil divisor.

Choose a map $\phi \in \Hom_R(\Fr{e}{R(\lceil (p^e - 1) \Delta \rceil)}, R)$.  By using \cite[Lemma 3.4]{HaraWatanabeFRegFPure}, we may identify $\phi$ with an element $f \in R( \lfloor (1 - p^e)(K_X + \Delta) \rfloor )$.  Choose $r > 0$ to be an integer such that $r(K_R + \Delta)$ is Cartier.  Thus
\[
f^r \in R( r \lfloor (1 - p^e)(K_X + \Delta) \rfloor) \subseteq R((1 - p^e)(r (K_X + \Delta))).
\]
Therefore we can view $f^r$ as a global section of
\[
\O_{\tld X}((1 - p^e) \pi^*(r(K_X + \Delta))) \subseteq \O_{\tld X}(r \lceil (1 - p^e) {1 \over r} \pi^*(r(K_X + \Delta))\rceil ).
\]
Write ${1 \over r} \pi^*(r(K_X + \Delta)) = K_{\tld X} - \Sigma a_i E_i$ (here we mean equality as $\bQ$-Weil divisors).  Note that not all of the $E_i$ are exceptional divisors; some correspond to components of the strict transform of $\Delta$.  Also note that $\Sigma a_i E_i = K_{\tld X} - \pi^*(K_X + \Delta)$.  Thus, we may view $f$ as a global section of
\[
\O_{\tld X}(\lceil (1 - p^e) ( K_{\tld X} - \Sigma a_i E_i ) \rceil) \subseteq \O_{\tld X}((1 - p^e) \lfloor K_{\tld X} - \Sigma a_i E_i \rfloor) \subseteq \O_{\tld X}( (1 - p^e) (K_{\tld X} - \lceil \Sigma a_i E_i \rceil - G)),
\]
which we view as $\O_{\tld X}( (1 - p^e) K_X - (1 - p^e)(\lceil \Sigma a_i E_i \rceil + G)).$
By using \cite[Lemma 3.4]{HaraWatanabeFRegFPure} in the opposite direction as before, we can view $f$ as some $\psi$ in the global sections of
\begin{align*}
& \sHom_{\O_{\tld X}}(\Fr{e}{\O_{\tld X}((1 - p^e)(\lceil \Sigma a_i E_i \rceil + G))}, \O_X) \\
\cong & \sHom_{\O_{\tld X}}(\Fr{e}{\O_{\tld X}(\lceil \Sigma a_i E_i \rceil + G)}, \O_{\tld X}(\lceil \Sigma a_i E_i \rceil + G)).
\end{align*}
At the level of the fraction field of $R$ (or simply outside of the exceptional set of $\pi$ and outside the support of $\Delta$), $\phi$ and $\psi$ induce the same map.  Therefore, considering the action of $\psi$ on global sections, we see that $\pi_* \O_{\tld X}(\lceil K_{\tld X} - \pi^*(K_X + \Delta) \rceil + G) = \pi_* \O_{\tld X}(\lceil \Sigma a_i E_i \rceil + G)$ is uniformly $(\Delta, F)$-compatible.
\end{proof}

In some sense we can view the ideals obtained in this way as a generalization of centers of log canonicity (in particular, any center of log canonicity may be obtained for an appropriate choice of $G$).  Also note that all of the ideals obtained this way are necessarily integrally closed, but there are test ideals which are not integrally closed.

In particular, if in a fixed positive characteristic we have a log resolution of $(X, \Delta)$, then the multiplier ideal computed from that resolution is $(\Delta, F)$-compatible.  The same argument works for adjoint ideals as well since we have a great deal of freedom in how we choose the divisor $G$.  We also note, as mentioned above, that for any exceptional divisor $E_i$ of $\pi$ with discrepancy $\leq -1$, we can choose $G$ so that $\lceil K_{\tld X} - \pi^*(K_X + \Delta) \rceil + G \geq -E_i$.  Thus we obtain the following result.

\begin{theorem}
\label{TheoremCentersOfLogCanonicityReduce}
Suppose that $(X, \Delta)$ is a pair where $X$ is a variety of finite type over a field of characteristic zero.  Further suppose that $K_X + \Delta$ is $\bQ$-Cartier.  Suppose that $Q \in X$ is a center of log canonicity corresponding to an exceptional divisor $E$ in some proper birational morphism $\pi : \tld X \rightarrow X$ (or a non-exceptional center as well).  Then after generic reduction to characteristic $p > 0$, the corresponding ideals $Q_p$ are uniformly $(X_p, \Delta_p)$-compatible.
\end{theorem}

\section{Results related to $F$-adjunction and conductor ideals}
\label{SectionResultsRelatedToFAdjunction}

In this section, we prove a number of local geometric properties of centers of $F$-purity.  There have been a number of related results in the past, see for example \cite{VassilevTestIdeals}, \cite{SchwedeSharpTestElements} and \cite[Theorem 4.1]{EnescuHochsterTheFrobeniusStructureOfLocalCohomology}.  Many of these results can also be proven directly, using the Fedder-type criteria, for sharply $F$-pure pairs $(R, \ba_{\bullet})$ under the hypothesis that $R$ is a quotient of a regular ring.

\begin{theorem}
\label{TheoremUnionOfCentersOfFPurityFormAnFPureSubscheme}
Suppose that $(R, \Delta, \ba_{\bullet})$  is sharply $F$-pure.  Then any finite (scheme-theoretic) union of centers of sharp $F$-purity for $(R, \Delta, \ba_{\bullet})$ form an $F$-pure subscheme.
\end{theorem}
\begin{proof}
A (scheme-theoretic) union of centers of sharp $F$-purity is, by definition, a reduced scheme such that each irreducible component is a center of sharp $F$-purity.  Let $I$ denote the reduced ideal defining this scheme.  Note that $I = \bigcap_{i=1}^n P_i$ where the $P_i$ are prime centers of sharp $F$-purity.  But then $I$ is uniformly $(\Delta, \ba_{\bullet}, F)$-compatible by Lemma \ref{EasyPropertiesOfUniformlyFCompatibleIdeals}, which implies that $R/I$ is $F$-pure by Corollary \ref{CorUniformlyFCompatibleDefinesFPureSubscheme}.
\end{proof}

\begin{remark}
 Actually, the same proof shows that the closure of any (possibly-infinite, scheme-theoretic) union of centers of sharp $F$-purity for $(R, \Delta, \ba_{\bullet})$ form an $F$-pure subscheme.  However, we do not know if the set of centers of sharp $F$-purity for $(R, \Delta, \ba_{\bullet})$ can be infinite (in the local case, it is always finite as we have shown).
\end{remark}




\begin{corollary}
\label{CorollaryUnionOfCentersOfLogCanonicityFPureType}
Suppose $(X, \Delta)$ is a pair over $\bC$ and $K_X + \Delta$ is $\bQ$-Cartier.  If $(X, \Delta)$ is of dense sharply $F$-pure type, then any (scheme-theoretic) union of centers of log canonicity also has dense $F$-pure type.  In particular, any such (scheme-theoretic) union has Du Bois singularities.
\end{corollary}
\begin{proof}
 The main statement is a direct result of Theorem \ref{TheoremUnionOfCentersOfFPurityFormAnFPureSubscheme}.  The statement about Du Bois singularities follows immediately from the fact that $F$-pure singularities are $F$-injective and from \cite[Theorem 6.1]{SchwedeFInjectiveAreDuBois}.
\end{proof}

\begin{theorem}
Suppose that $(R, \Delta, \ba_{\bullet})$ is sharply $F$-pure.  Then any finite (scheme-theoretic) intersection of centers of sharp $F$-purity is a (scheme-theoretic) union of centers of sharp $F$-purity, and so the (scheme-theoretic) intersection cuts out an $F$-pure subscheme.
\end{theorem}
\begin{proof}
This immediately follows from Lemma \ref{EasyPropertiesOfUniformlyFCompatibleIdeals}.  Alternately, if $R$ is a quotient of a regular local ring and $\Delta = 0$, then one can also prove the result using Proposition \ref{PropFedderCriterionForCompatible}.
\end{proof}

The following proposition shows that centers of sharp $F$-purity themselves force the existence of other centers of sharp $F$-purity.

\begin{proposition}
\label{PropositionCentersOfFPurityForQuotientsInduceCenters}
Suppose that $(R, \Delta, \ba_{\bullet})$ is a triple and that $I$ is a uniformly $(\Delta, \ba_{\bullet}, F)$-compatible ideal.  Further suppose that $P \supset I$ corresponds to a center of $F$-purity for $R/I$.  Then $P$ is a center of sharp $F$-purity for $(R, \Delta, \ba_{\bullet})$ as well.
\end{proposition}
\begin{proof}
Choose an arbitrary $e > 0$, an element $a \in \ba_{p^e - 1}$, and a map $\gamma : \Fr{e}{R(\lceil (p^e - 1) \Delta \rceil)} \rightarrow R$.
Since $I$ is uniformly $(\Delta, \ba_{\bullet}, F)$-compatible, we have a commutative diagram
\[
\xymatrix{
\Fr{e}{R} \ar[d] \ar[r]^-{\Fr{e}{(\times a)}} & \Fr{e}{R}\ar[r] & \Fr{e}{R(\lceil (p^e - 1) \Delta \rceil)} \ar[r]^-{\gamma} & R \ar[d] \\
\Fr{e}{R/I} \ar[rrr]^{\phi} &&& R/I
}
\]
where the vertical arrows are the natural maps.  Note that there exists a commutative diagram
\[
\xymatrix{
\Fr{e}{R/I} \ar[r]^-{\phi} \ar[d] & R/I \ar[d] \\
\Fr{e}{R/P} \ar[r]^-{\psi} & R/P.
}
\]
But this gives us a commutative diagram
\[
\xymatrix{
\Fr{e}{R} \ar[d]^{\alpha} \ar[r]^-{\Fr{e}{(\times a)}} & \Fr{e}{R} \ar[r] & \Fr{e}{R(\lceil (p^e - 1) \Delta \rceil)} \ar[r]^-{\gamma} & R \ar[d]^{\beta} \\
\Fr{e}{R/P} \ar[rrr]^{\psi} & & & R/P.
}
\]
Note that $\ker(\alpha) = \Fr{e}{P}$ and $\ker{\beta} = P$, so we see that $\gamma(\Fr{e}{(a P)}) \subseteq P$, which proves that $P$ is a center of sharp $F$-purity for $(R, \Delta, \ba_{\bullet})$ since $\gamma$ was arbitrary.
\end{proof}

\begin{remark}
 In the statement of Proposition \ref{PropositionCentersOfFPurityForQuotientsInduceCenters}, $P$ need not be prime (or even radical), it simply must be an ideal containing $I$ which corresponds to a uniformly $F$-compatible ideal of $R/I$.  The proof is unchanged.
\end{remark}

\begin{remark}
 The converse to the previous proposition is not true.  In particular, there exists a ring $R$, an ideal $I \subset R$ and a prime $P \supset I$ such that $P$ is a center of $F$-purity for $R$ but that $P$ does not correspond to a center of $F$-purity for $R/I$.  For example, consider the ring
\[
 R = k[a,b,c]/(a^3 + abc - b^2) = k[xy, x^2y, x-y] \subset k[x,y]
\]
where $k$ is a perfect field of positive characteristic.  (Geometrically, this example is obtained by taking the two axes in $\bA^2$ and gluing them together).

It is easy to see, using Fedder's criterion, that this ring is $F$-pure; see \cite{FedderFPureRat}.  The singular locus is defined by the height one prime ideal $(a,b)$.  Thus since $\dim R = 2$, the ideal $(a,b)$ is a center of $F$-purity (since strongly $F$-regular rings are normal).  Also note that $R/(a,b) \cong k[c]$ is regular (and so it has no non-trivial centers of $F$-purity).  On the other hand, I claim that the ideal $(a,b,c)$ is also a center of $F$-purity.  To see this, simply note that for any $f \in (a,b,c)$, we have that
\[
f (a^{3} + abc - b^{2})^{p^e - 1} \in (a^{p^e}, b^{p^e}, c^{p^e})
\]
for $e > 0$ simply because of the middle term $(abc)$.

On the other hand, let $I_C$ be the conductor ideal of $R$ in its normalization $R^N = k[x,y]$.  Inside $R^N$, the conductor ideal simply cuts out the two axes, that is $I_C = (xy)$.  It is then interesting to note that the origin $(x,y)$ is a center of sharp $F$-purity for the pair $(R^N, I_C) = (k[x,y], \Div(xy))$.
\end{remark}

The following corollary is closely related to Kawamata's subadjunction theorem, \cite{KawamataSubadjunction2}.  Note that we do not need to assume that the ambient ring $R$ is strongly $F$-regular (the characteristic $p$ analogue of Kawamata log terminal).  The following result generalizes \cite[Theorem 4.7]{AberbachEnescuStructureOfFPure} to triples $(R, \Delta, \ba_{\bullet})$, to the non-local case, and to the situation where $R$ is not (necessarily) the quotient of a regular ring.

\begin{corollary} \cite[Theorem 4.7]{AberbachEnescuStructureOfFPure}
\label{TheoremMaximalCenterOfFPuritySubadjunction}
Suppose that $(R, \Delta, \ba_{\bullet})$ is a triple and $P$ is a center of sharp $F$-purity for $(R, \Delta, \ba_{\bullet})$ that is maximal (as an ideal) among the centers of $F$-purity for $(R, \Delta, \ba_{\bullet})$ with respect to ideal containment.  Then $R/P$ is strongly $F$-regular.
\end{corollary}
\begin{proof}
Suppose that $R/P$ is not strongly $F$-regular, then there exists a prime $Q \supsetneq P$ such that $Q/P$ is a center of $F$-purity for $R/P$ (for example, $Q$ could correspond to a minimal prime of the non-strongly $F$-regular locus of $R/P$).  But then $Q$ is a center of sharp $F$-purity for $(R, \Delta, \ba_{\bullet})$ as well, by Proposition \ref{PropositionCentersOfFPurityForQuotientsInduceCenters}, contradicting the maximality of the choice of $P$.
\end{proof}

\begin{remark}
Note that the previous result is not very interesting unless the triple $(R, \Delta, \ba_{\bullet})$ is sharply $F$-pure.  This is because in a local ring $(R, \bm)$ such that $(R, \Delta, \ba_{\bullet})$ is not sharply $F$-pure, the maximal height center of $F$-purity for $(R, \Delta, \ba_{\bullet})$ is the maximal ideal.
\end{remark}

%

Suppose that $R$ is a ring and $R^N$ is its normalization.  Since the height one associated primes of the conductor ideal of $R \subset R^N$ are among the minimal primes of the non-strongly $F$-regular locus, we see that the intersection of the height one associated primes of the conductor ideal are uniformly $F$-compatible.  This suggests that we might ask the following question: ``Is the conductor uniformly $F$-compatible as an ideal in $R$?"  Something very similar was done in \cite[Proposition 1.2.5]{BrionKumarFrobeniusSplittingMethods}, where they only considered splittings.  Therefore, for the convenience of the reader, we will modify their proof into our setting.

\begin{proposition}
\label{PropConductorIsFCompatible}
Given a reduced $F$-finite ring $R$ with normalization $R^N$, the conductor ideal of $R$ in $R^N$ is uniformly $F$-compatible.
\end{proposition}
\begin{proof}
The conductor ideal $I$ can be defined as ``the largest ideal $I \subseteq R$ that is simultaneously an ideal of $R^N$''.  It can also be described as
\[
 I := \Ann_R R^N/R = \{x \in R | x R^N \subseteq R\}.
\]
Following the proof of \cite[Proposition 1.2.5]{BrionKumarFrobeniusSplittingMethods}, consider $\phi \in \Hom_R(F^e_* R, R)$.  Notice, that by localization, $\phi$ extends to a map on the total field of fractions (which contains $R^N$).  We will abuse notation and also call this map $\phi$ (since it restricts to $\phi : F^e_* R \rightarrow R$).  Now choose $x \in F^e_* I$ and $r \in R^N$.  Then $\phi(x) r = \phi(x r^{p^e}) \in \phi(F^e_* R) \subseteq R$.  Thus $\phi(x) \in I$ as desired.
\end{proof}

One might ask even more generally if every map $\phi : F^e_* R \rightarrow R$ extends to a map on the normalization?  In particular, if $R$ is $F$-split, then is its normalization also $F$-split?  This second statement is found in \cite[Exercise 1.2.E(4)]{BrionKumarFrobeniusSplittingMethods}, and the proof also implies the first statement.

\begin{proposition}
\label{PropEveryMapExtends}
For a reduced $F$-finite ring $R$, every map $\phi : F^e_* R \rightarrow R$, when viewed as a map on total field of fractions, restricts to a map $\phi' : F^e_* R^N \rightarrow R^N$ on the normalization.
\end{proposition}
\begin{proof}
We follow the hint for \cite[Exercise 1.2.E(4)]{BrionKumarFrobeniusSplittingMethods}.  For any $x \in R^N \in K(R)$, we wish to show that $\phi(x) \in R^N$.  First we show that we can reduce to the case of a domain.  We write $R \subseteq K(R) = K_1 \oplus \dots \oplus K_t$ as a subring of its total field of fractions.  Since each minimal prime $Q_i$ of $R$ is uniformly $F$-compatible, it follows that $\phi$ induces a map $\phi_i : F^e_* R/Q_i \rightarrow R/Q_i$ for each $i$.  Notice that the normalization of $\Spec R$ is a disjoint union of components (each one corresponding to a minimal prime of $R$), and the $i$th component is equal to $\Spec (R/Q_i)^N$.  Thus, since we are ultimately interested in $\phi'$ restricted to each $(R/Q_i)^N$, it is harmless to assume that $R$ is a domain.

Suppose that $I$ is the conductor and consider $I\phi(x)$.  For any $z \in I$, $z\phi(x) = \phi(z^{p^e}x) \in \phi(F^e_* I) \subseteq I$ (notice that $z^{p^e}x \in I$ since $I$ is an ideal of $R^N$).  More generally, $z\phi(x)^m = z\phi(x)(\phi(x))^{m-1} \subseteq I (\phi(x))^{m-1}$ which implies that $I\phi(x)^m \subseteq I \phi(x)^{m-1}$, and so by induction $I \phi(x)^m \subseteq I \subset R$ for all $m > 0$.  This implies that for every $c \in I \subseteq R$ we have that $c \phi(x)^m \in I \subseteq R$.  Therefore $\phi(x)$ is integral over $R$ by \cite[Exercise 2.26(iv)]{HunekeSwansonIntegralClosure}.
\end{proof}

\begin{remark}
With notation as above, notice that $\phi : F^e_* R \rightarrow R$ sends $1$ to $1$ if and only if $\phi' : F^e_* R^N \rightarrow R^N$ sends $1$ to $1$.  This is because the two maps agree at the level of the total field of fractions.
\end{remark}

If $R$ is \stwo{} and seminormal (note that $F$-pure schemes are seminormal) then the conductor ideal is radical and height one in both $R$ and $R^N$, see \cite[Lemma 1.2]{TraversoPicardGroup} and \cite[Lemma 7.4]{GrecoTraversoSeminormal}.  Thus it follows that the conductor ideal corresponds to a reduced integral effective divisor $C$ on $X^N = \Spec R^N$. 
Now note that if $R$ is $F$-pure, then some surjective map $\phi \in \Hom_R(F^e_* R, R)$ extends to a $R^N$-linear map $\phi' : F^e_* R^N \rightarrow R^N$ by Proposition \ref{PropEveryMapExtends}.  Since $1 \in \phi(F^e_* R)$, we also have $1 \in \phi'(F^e_* R^N)$ and so $R^N$ is also $F$-pure.  Since the conductor ideal is uniformly $F$-compatible, $\phi'$ restricts to a map $\phi' : F^e_* I_C \rightarrow I_C$.  Applying $\Hom_{R^N}(I_C, \blank)$ we obtain a map $\phi'' : F^e_* R^N((p^e - 1)C) \rightarrow R^N$.  In fact, $\phi''$ restricts to $\phi$ since the two maps agree at the level of the total field of fractions.   It follows that the pair $(R^N, C)$ is also sharply $F$-pure.

\begin{corollary}
Suppose that $X$ is an affine \stwo{} scheme of finite type over a field of characteristic zero, and suppose that $X$ has dense $F$-pure type.  If $K_{X^N} + C$ is $\bQ$-Cartier then $X$ is weakly semi-log canonical.
\end{corollary}
\begin{proof}
The proof of this result follows from the main result of \cite{HaraWatanabeFRegFPure}.  One simply reduces both $X$ and its normalization $X^N$ to characteristic $p > 0$.
\end{proof}

\begin{corollary}
Suppose that $(X, \Delta)$ is a pair of finite type over a field of characteristic zero, and further suppose that $(X, \Delta)$ has dense sharply $F$-pure type.  Suppose that $W \subset \Spec R$ is a scheme-theoretic union of centers of log canonicity that is also \stwo.  Set $W^N$ to be the normalization of $W$ and set $C_W$ is the divisor corresponding to the conductor on $W^N$.  If $K_{W^N} + C_W$ is $\bQ$-Cartier then $W$ is weakly semi-log canonical.
\end{corollary}

\section{Further Questions}

There are a number of open questions that one can ask about this theory.





\begin{question}
Is the integral closure of a uniformly $F$-compatible ideal still uniformly $F$-compatible?  What about other ideal closures?
\end{question}


\begin{question}
 Is there a generalization of log canonical singularities (without the normal or $\bQ$-Cartier hypotheses) and a purely characteristic zero proof of the following?
 \vskip4pt \hskip-12pt
 {\bf Conjecture: } \emph{Suppose that $(X, \Delta)$ is a log canonical pair.  Then any scheme-theoretic union of centers of log canonicity is ``generalized log canonical''. }
 \vskip4pt \hskip-12pt
 Recent work of Hacon and de Fernex, see \cite{HaconDeFernex}, gives a definition of normal log canonical singularities without the hypothesis that $K_X + \Delta$ is $\bQ$-Cartier, so this may be a good place to start.
\end{question}




\providecommand{\bysame}{\leavevmode\hbox to3em{\hrulefill}\thinspace}
\providecommand{\MR}{\relax\ifhmode\unskip\space\fi MR}
\providecommand{\MRhref}[2]{%
  \href{http://www.ams.org/mathscinet-getitem?mr=#1}{#2}
}
\providecommand{\href}[2]{#2}

\end{document}